\documentclass{amsart}
\usepackage{bbm}
\usepackage{bbding}
\usepackage[mathcal]{euscript}
\usepackage{graphicx}
\usepackage{amsthm}
\usepackage{calc}
\usepackage{mathrsfs,dsfont}
\usepackage{CJK,fancyhdr,amscd}
\usepackage{anysize} 
\usepackage{amsmath,amssymb,amsfonts}
\usepackage{epsfig,enumerate}
\usepackage{latexsym}
\usepackage{indentfirst, latexsym}
\usepackage{graphics}
\usepackage[all,poly,knot]{xy}
\usepackage[colorlinks,linkcolor=blue,anchorcolor=blue,citecolor=blue,pagebackref]{hyperref}
\usepackage{geometry}
\numberwithin{equation}{section}
\newtheorem{theorem} {Theorem} [section]
\newtheorem{proposition}[theorem]{Proposition}
\newtheorem{corollary}  [theorem]     {Corollary}
\newtheorem{lemma}  [theorem]     {Lemma}

\theoremstyle{definition}
\newtheorem{definition}  [theorem]     {Definition}

\newcommand{\eps}{\varepsilon}

\renewcommand*\backref[1]{}
\renewcommand*\backrefalt[4]{ \ifcase #1 \or (cited on page #2) \else (cited on pages #2) \fi}

\begin{document}

\title{On Hermitian manifolds whose Chern connection is Ambrose-Singer}

\author{Lei Ni}
\address{Lei Ni. Department of Mathematics, University of California, San Diego, La Jolla, CA 92093, USA}
\email{leni@ucsd.edu}

\author{Fangyang Zheng} \thanks{The research is partially supported by NSFC grants \# 12071050  and 12141101, Chongqing grant cstc2021ycjh-bgzxm0139, and is supported by the 111 Project D21024.}
\address{Fangyang Zheng. School of Mathematical Sciences, Chongqing Normal University, Chongqing 401331, China}
\email{20190045@cqnu.edu.cn; \ franciszheng@yahoo.com }

\subjclass[2010]{53C55 (primary), 53C05 (secondary)}
\keywords{Hermitian manifold; Chern connection; Ambrose-Singer connections; Hermitian (locally) homogeneous manifolds}

\begin{abstract}
We consider the class of compact Hermitian manifolds whose Chern connection is Ambrose-Singer, namely, it has parallel  torsion and curvature. We prove structure theorems for such manifolds.
\end{abstract}

\subjclass[2010]{ 53C55 (Primary), 53C05 (Secondary)}
\keywords{Ambrose-Singer connection, Chern connection, holonomy algebra/group, generalized Calabi-Yau, Chern flat, locally homogeneous Hermitian manifolds}

\maketitle

\markleft{Ni and Zheng}
\markright{On CAS manifolds}

\tableofcontents

\section{Introduction and main results}

Generalizing Cartan's characterization of symmetric Riemannian manifolds, W. Ambrose and I. M. Singer obtained in 1958 \cite{AS} their celebrated theorem which states that a complete, simply-connected Riemannian manifold is (Riemannian) homogeneous (meaning that its isometry group acts transitively) if and only if it admits a metric connection with parallel torsion and curvature. In view of this result, a connection is said to be {\em Ambrose-Singer,}  if it has parallel torsion and curvature with respect to itself. More specifically, the result of \cite{AS} says that a complete Riemannian manifold is locally homogenous (meaning its universal cover is homogeneous) if and only if it admits a metric connection which is Ambrose-Singer. Similar results for affine connections (without  a metric) was obtained around the same time by H. C. Wang \cite{Wang-58}, whose coverage  can  be found on page 262 of Vol.I of \cite{KN} and Theorem 2.8 of Ch.10 of \cite{KN}.

When the manifold is Hermitian (or more generally, almost Hermitian), one naturally restricts the consideration to {\em Hermitian connections,} meaning a connection $\nabla$ on the manifold that is both metric ($\nabla g=0$) and almost complex ($\nabla J=0$), where $g$ is the Hermitian metric and $J$ the almost complex structure. The Ambrose-Singer Theorem has a natural extension to the (almost) Hermitian version, proved by Sekigawa \cite{Sekigawa} in 1978. It states that a complete, simply-connected Hermitian manifold is (Hermitian) homogeneous (meaning its group of holomorphic isometries acts transitively) if and only if it admits a Hermitian Ambrose-Singer connection.

If one drops the completeness and simply-connectedness assumption, the concept of local homogeneity can be defined via local isometries, and the correlation between local homogeneity and the existence of Ambrose-Singer connections is still valid. In the literature there are extensive studies of the correspondence between the geometry of locally homogeneous manifolds and algebraic models formed by the torsion and curvature of an AS connection. We refer the readers to the papers \cite{Singer}, \cite{Kiricenko}, \cite{Tricerri}, \cite{TV}, \cite{NT}, \cite{CN} and the references therein for more details in this direction. In this article, we are only interested in compact complex manifolds, so being locally homogenous for us simply means the universal cover is homogeneous.

On the other hand  a Hermitian manifold admits some canonical/unique Hermitian connections, including the Chern connection $\nabla^c$, the Strominger (also known as Bismut, cf.\,\cite{WYZ} for some details) connection $\nabla^s$, and their combinations, the so-called $t$-Gauduchon connections $\nabla^{t} = (1-t)\nabla^c + t\nabla^s$, where $t\in {\mathbb R}$. Here we study the question:
{\it
What  compact Hermitian manifolds $(M^n,g)$  will have their Chern (or Strominger, or $t$-Gauduchon) connection to be Ambrose-Singer?}

As the Chern/Strominger-Bismut connection is Levi-Civita if and only if the manifold is K\"ahler, we expect that the above class of manifolds  forms a special class of locally homogenous manifolds which is subject to explicit descriptions. For convenience in future discussions, let us introduce the following terminology:

\begin{definition}
A {\em CAS manifold} is a compact Hermitian manifold with its Chern connection being Ambrose-Singer, namely, the Chern connection has parallel torsion and curvature. Dropping the compactness assumption we say such a Hermitian manifold admits a {\em CAS structure.}
\end{definition}

One can define {\em SAS} or {\em $t$-GAS manifolds} similarly for compact Hermitian manifolds whose Strominger/Bismut connection is Ambrose-Singer. The project is to understand the set of all CAS (or SAS, or $t$-GAS) manifolds. In this paper we  focus on the CAS case, and leave the other types to a future study.

For homogeneous complex manifolds and homogeneous K\"ahler manifolds (namely a K\"ahler manifold admits a transitive action by a holomorphic isometry group), much more have been known since the celebrated works of H. C. Wang \cite{Wang}, Borel \cite{Borel}, Tits \cite{Tits} etc. In particular, \cite{Tits} (cf. see also \cite{GR}) proved that any compact  homogeneous complex manifold  has a unique fibering with parallelizable fibre over a D-space (which is a simply-connected homogeneous complex manifold obtained by a complex Lie group module a parabolic subgroup).  Some   classifications of the lower dimension cases can also be found there. The simply-connected compact homogeneous complex manifolds were characterized by H. C. Wang. One can find a detailed classification/study of three dimensional homogeneous complex manifolds in \cite{Win}. It was further proved in \cite{GR} that a compact homogenous complex manifold  is a holomorphic fiber bundle over a homogenous projective manifold with a complex parallelizable fiber.

The first observation is that all CAS manifolds are Chern K\"ahler-like, meaning that the curvature tensor $R$ of the Chern connection $\nabla$ obeys all the symmetry conditions enjoyed by the curvature tensor of a K\"ahler metric. Moreover, by \cite[Theorem 3]{YZ}, any compact Chern K\"ahler-like manifold is balanced, meaning that $d(\omega^{n-1})=0$ where $\omega$ is the K\"ahler form. In particular, for the case $n=\dim_{\mathbb{C}}(M)=2$, the metric is K\"ahler, thus the manifold is a locally Hermitian symmetric space. So the nontrivial study of CAS manifolds starts with dimension $3$.

By a classical theorem of Boothby \cite{Boothby}, compact Hermitian manifolds with flat Chern connection are exactly quotients of complex Lie groups equipped with left invariant Hermitian metrics, in particular the Chern connection has parallel torsion. That is, {\em any compact Chern flat manifolds are always CAS}, so we have
$$ \{ \mbox{Chern\, flat} \} \subseteq  \{ \mbox{CAS} \} \subseteq  \{ \mbox{Chern\, K\"ahler-like}\}.$$

Given a CAS manifold $(M^n,g)$, if $g$ is K\"ahler, then the Chern connection coincides with the Levi-Civita connection, so the manifold is a locally Hermitian symmetric space. On the other hand, any compact locally Hermitian symmetric space is certainly CAS (and $t$-GAS for any $t$ for that matter), so our quest here is really to understand the non-K\"ahler ones.

A purpose of this article is to give a classification of CAS manifolds in dimensions $3$ and $4$:

\begin{theorem}\label{thm-dim3}
Let $(M^3,g)$ be a compact  CAS manifold, namely a compact Hermitian manifold whose Chern connection has parallel torsion and curvature. Then it is either K\"ahler (hence a locally Hermitian symmetric space) or Chern flat.
\end{theorem}
Three dimensional CAS examples include the Iwasawa manifolds.
We remark that there exist plenty  of complex  structure on the (complex) three dimensional simply connected manifold  $\mathbb{S}^3\times\mathbb{S}^3$ \cite{CE}. It in fact can be made complex homogenous, and it admits many homogenous Hermitian structure via averaging,  since the compact group $\mathsf{U}(2)\times \mathsf{U}(2)$ acts on it holomorphically and transitively. The above theorem asserts that the AS connection of the associated homogenous Hermitian structure can not coincide with the Chern connection since it is well known that this complex manifold can not admit a K\"ahler structure and any compact complex Lie group must be a complex torus.

\begin{theorem}\label{thm-dim4}
Let $(M^4,g)$ be a CAS manifold, namely a compact Hermitian manifold whose Chern connection has parallel torsion and curvature. Then it is either K\"ahler (hence a locally Hermitian symmetric space), or Chern flat, or the universal cover is holomorphically isometric to the product $C\times G$, where $C$ is a complex space form of dimension $1$ and $G$ a complex Lie group of dimension $3$ (with left invariant metric).
\end{theorem}

Similarly, the above result asserts that the complex homogenous manifold $\mathbb{S}^3\times \mathbb{S}^5$ could not admit a CAS structure.
In this paper we call a CAS manifold {\em trivial} if each of its de Rham factor of its universal cover is either K\"ahler (hence a Hermitian symmetric space) or Chern flat (namely a complex Lie group with a left-invariant Hermitian metric). The above two theorems implies that in dimension $n\leq 4$, CAS manifolds are all trivial.

For general dimensions, we give a characterization of when a CAS manifold  splits off K\"ahler de Rham factors (Theorem \ref{thm-N0} in \S 3), and generalize the results in dimensions $3$ - $5$ to higher dimensions in terms of the codimension of the {\em image distribution} of the Chern torsion (Theorem \ref{thm-imagecodim} in \S 6). Here the de Rham factor means the product factor splitting on the universal cover induced by the holonomy principle.  We will also prove the following:

\begin{theorem} \label{thm-Ricciflat}
Let $(M^n,g)$ be a Hermitian manifold  with a  CAS structure. Assume further that it is without any K\"ahler de Rham factor. Then the universal cover of $M$ admits a parallel holomorphic $(n, 0)$ form, and $(M,g)$ has zero Chern Ricci curvature. In particular, the action of any restricted holonomy group element $h$ (with respect to the Chern connection) on $T^{1,0}M$ splits as $\tilde{h}\oplus \operatorname{id}$ with  $\tilde{h}\in \mathsf{Sp}(k)$ for some $1\le k\le[(n-1)/2]$, with $k\ge 2$.
\end{theorem}

An interesting consequence of the above result is  that Hermitian manifolds with a CAS structure and without any K\"ahler de Rham factor can be viewed as generalized Calabi-Yau manifolds. As a consequence, we have

\begin{corollary} \label{cor-homog}
Let $(M^n,g)$ be a CAS manifold. If $M^n$ is a compact homogeneous (or almost homogeneous) complex manifold, then $(M^n,g)$ is trivial. In particular, its universal cover is holomorphically isometric to the product $G\times N^{k}$, with $0\leq k\leq n $, where $G$ is a (connected, simply-connected) complex Lie group equipped with a left-invariant metric, and $N^k$ is a compact Hermitian symmetric space of dimension $k$.
\end{corollary}
Recall that a compact homogeneous (almost homogeneous) complex manifold is one where the group of automorphisms (biholomorphisms), which is a complex Lie group, acts transitively (or having an open orbit).

The corollary asserts that none of the simply-connected  Calabi-Eckmann manifolds $\mathbb{S}^{2k+1}\times \mathbb{S}^{2l+1}$, with $k, l\ge 1$, which are complex homogenous spaces,  can admit a CAS structure.

In the mean time, in dimension 5, we provide in    Theorem \ref{thm-dim5}  a structural characterization to any possibly non-trivial example in dimension $5$. It is our hope that this result is instrumental in constructing non trivial examples starting dimension 5. Such a fivefold, if exists, should be a complex $1$-torus bundle over a holomorphic symplectic $4$-manifold. This will give an interesting non K\"ahler generalization of Calabi-Yau manifolds, a subject which has attracted attentions due to considerations in mathematical physics \cite{TY}.  Please refer to \cite{QW} for a more recent related construction and references in these papers for more comprehensive coverage on the subject.  This example, if exists,  also shows that CAS class is strictly larger  than the Chern flat ones. On the other hand, if all CAS manifolds whose universal cover does not contain any K\"ahler factors must be Chern flat, it would also provide a nice characterization of CAS manifolds in high dimensions, although this assertion will reduce the scope of interests of CAS manifolds.



The article is organized as follows. In the next section, we  briefly recall the construction in Ambrose-Singer Theorem as well as its Hermitian version, Sekigawa Theorem. In \S 3, we  analyze the properties of CAS manifolds in general dimensions, and characterize K\"ahler de Rham factors of a CAS manifold (Theorem \ref{thm-N0}). Theorem \ref{thm-dim3} and \ref{thm-dim4} shall be proved in \S 4. In \S 5 we discuss the $5$-dimensional situation and prove Theorem \ref{thm-dim5}. In \S 6 we  partially generalize the statements of Theorems \ref{thm-dim3} and \ref{thm-dim4} into higher dimensions (Theorem \ref{thm-imagecodim}) and  prove Theorem \ref{thm-Ricciflat}. In \S 7 we  prove Corollary \ref{cor-homog} and have some general discussion of CAS versus homogeneity.

We should remark that on a homogeneous Riemannian manifold $M^n$,  there might be more than one Ambrose-Singer connections, namely the Ambrose-Singer connection  may not be unique. If we consider the set of all Ambrose-Singer connections on $M^n$, it is also not clear what kind of structure  this set possess. This is related to the fact that there might be multiple subgroups of the isometry group which act transitively on the manifold. One might ask {\it for  what kind of complete, simply-connected Riemannian (Hermitian) homogeneous manifolds is its metric (Hermitian) Ambrose-Singer connection  unique?}


\vspace{0.3cm}

\section{Ambrose-Singer and Sekigawa Theorems}

In this section, we recall the main result  and some arguments in Ambrose-Singer Theorem \cite{AS} and its Hermitian version \cite{Sekigawa}.

Let $(M^n,g)$ be a complete, simply-connected Riemannian manifold. Denote by $F(M)$ the bundle of orthogonal frames on $M^n$. It is a principal bundle with structure group $O(n)$. Let $\pi : F(M)\rightarrow M$ be the projection map. Any point $b\in F(M)$ is in the form $b=(x; \eps_1, \ldots , \eps_n)$, where $x=\pi (b) \in M$ and $\{ \eps_1, \ldots , \eps_n \}$  is an orthonormal basis of the tangent space $T_xM$. For  $b\in F(M)$, denote by $V_b$ the kernel of $d\pi$ at $b$, so $V$ is the vertical distribution of the bundle $\pi$.

The frame bundle $F(M)$ naturally admits a global tangent frame $\{ E_i, E_{jk}\}$, where $1\leq i\leq n$, $1\leq j<k\leq n$, so that $E_{jk} \in V$ and  $E_i(b)=\eps_i$. Using this natural frame to be an orthonormal frame, we get a Riemannian metric $\widehat{g}$ on $F(M)$, and $\pi$ becomes a Riemannian submersion.

Now if $H$ is a subgroup of the isometry group $I(M)$ such that $H$ acts transitively on $M$, then  for any $b=(x; \eps_1, \ldots , \eps_n) \in F(M)$, we have a smooth map
\begin{eqnarray*}
\Psi_b: && H \rightarrow F(M)\\
&& h \mapsto hb = (h(x); h_{\ast}\eps_1, \ldots , h_{\ast}\eps_n)
\end{eqnarray*}
Let $T_eH$ be the tangent space of the Lie group $H$ at its unit element $e$. Since $\Psi_b(e)=eb=b$, one can define distributions in $F(M)$ by
$$ P_b = d\Psi_b (T_eH), \ \ \ \ \ \ Q_b = P_b \cap \big(P_b\cap V_b \big)^{\perp} $$
where $\perp$ is with respect to the Levi-Civita connection of $\widehat{g}$ on $F(M)$.  Clearly, $Q$ is a subbundle of the tangent bundle of $F(M)$ such that $Q\oplus V = TF(M)$, that is, $Q$ gives a horizontal distribution of $\pi$. As is well known, metric connections on $M$ are in one one correspondence with horizontal distributions on $F(M)$, so the above $Q$ gives a metric connection $\nabla$ on $M$. It is proved in \cite{AS} that the connection is AS, namely, its torsion and curvature are parallel with respect to the connection itself.

Conversely, if we start with an AS metric connection $\nabla$, corresponding to a horizontal distribution $Q$ on $F(M)$. Then fix any $b_0\in F(M)$, and denote by $H\subseteq F(M)$ the subset of points in $F(M)$ that can be connected to $b_0$ by  piecewise smooth horizontal paths. The AS condition leads to the fact that $H$ is a Lie group, acting on $M$ transitively as isometries. This is the basic argument in \cite{AS}.

From the above argument, we see that AS connections on $M$ are in one to one correspondence with conjugacy classes of connected Lie subgroups $H$ of $I(M)$ which act transitively on $M$. So AS connections on $M$ may not be unique, and to understand the set of admissible AS connections, the concern is to determine the conjugacy classes of connected Lie subgroups of $I(M)$ that act transitively on $M$.

Note that amongst all the AS connections, the one corresponding to the identity component $I_0(M)$ is uniquely determined, and we will call it the {\em canonical AS connection} of $M$. It has `maximum symmetry', and it would certainly be interesting to study its geometric significance.  We will leave the exploration of the set of AS connections to a future project.

\vspace{0.1cm}

Now let us switch gear to consider the Hermitian case. Let $(M^n,g)$ be a complete, simply-connected Hermitian manifold, and denote by $J$ its almost complex structure. As before, we have the frame bundle $F(M)$ which is a principal $O(2n)$-bundle over $M$. Denote by $\widetilde{F}(M) \subseteq F(M)$ the subbundle of all unitary frames of $M$, namely, points in $F(M)$ in the form
$$ b= (x; \,\eps_1, \ldots , \eps_n, J\eps_1, \ldots , J\eps_n) \in F(M).  $$
$\widetilde{F}(M)$ is a $U(n)$-bundle over $M$. Following the argument of \cite{AS}, if $H$ is a group of holomorphic isometries acting transitively on $M$, then for $b\in \widetilde{F}(M)$ the map $\Psi_b$ clearly sends $H$ into $\widetilde{F}(M)$, so $Q$ constructed as before would be a horizontal distribution for $\widetilde{\pi}: \widetilde{F}(M) \rightarrow M$, thus corresponding to  a Hermitian connection on $M$.

Conversely, given a Hermitian AS connection, it corresponds to a horizontal distribution $Q$ in $\widetilde{F}(M)$. By letting $H\subseteq \widetilde{F}(M)$ be the set of all points that can be connected to a fixed point $b_0\in \widetilde{F}(M)$ by piecewise smooth horizontal paths, one gets a Lie group acting transitively on $M$ as isometries. The fact $H \subseteq \widetilde{F}(M)$ shows that the elements of $H$ also act holomorphically.

In summary, for a complete, simply-connected Hermitian manifold $M^n$, the Hermitian AS connections on $M$ are in one one correspondence with the conjugacy classes of connected Lie subgroups in $\widetilde{I}(M)$, where $\widetilde{I}(M)$ is the intersection of $I(M)$ with the group of automorphisms of $M$ (i.e., biholomorphisms from $M$ onto itself)

Again, the Hermitian AS connection corresponding to the identity component $\widetilde{I}_0(M)$ has `maximum symmetry' and is of particular interest. It is uniquely determined and will once again be called the {\em canonical Hermitian AS connection}.  The Hermitian analogue, namely the uniqueness of Hermitian AS connections, or more generally, the structure and property of the set of all Hermitian AS connections on a given (compact, locally) Hermitian homogeneous space should be a topic worthy of further exploration, and we intend to investigate it in the future.

We  remark that in \cite{AS} the authors worked with the tensor $S_XY:=\nabla_XY-\nabla^{LC}_X Y$ with $\nabla^{LC}$ being the Levi-Civita connection with respect to the Hermitian metric, and $\nabla$ being the Ambrose-Singer connection. It is easy to see that the torsion of $\nabla$, which we will denote by  $T(X, Y)$ or $T_{X, Y}$, can be expressed in $S$ as $S_X Y-S_Y X$. On the other hand, $T$ completely determines $S$ by the formula:
\begin{equation} \label{eq:ts}
\langle S_U V, W\rangle =\frac{1}{2}\left( \langle T(U, V), W\rangle+\langle T(W, U), V\rangle
-\langle T(V,W), U\rangle\right).
\end{equation}
Here we mainly work with the torsion $T$. It is easy to see that $T$ is  parallel  if and only $S$ is parallel. Hence $T$ is parallel with respect to an Ambrose-Singer connection. When $\nabla$ is the Chern connection it is also useful to note that  $A_X Y:=L_X Y- \nabla_X Y$ commutes with $J$ if $X$ is {\em real holomorphic} (namely $L_XJ=0$ or equivalently, $L_X\circ J = J \circ L_X$). Here and below $L_X$ stands for the Lie derivative. On a Riemannian manifold (with respect to the Levi-Civita connection) the operator $A_X$ plays a role as the `infinitesimally rotational part' of the one parameter family of isometric groups generated by a Killing vector field $X$ (cf. \cite{Kostant}). In a recent article \cite{Ust}, the so-called torsion twisted connection, which is defined as $\nabla^T_X Y\doteqdot \nabla_X Y -T(X, Y)$,  was found its relevance in the study of generalized flows on Hermitian manifolds. It is easy to see that $\nabla^T_X Y=-A_Y X$.

\section{Chern Ambrose-Singer manifolds}

In this section we will examine the general properties of CAS manifolds. Let  $(M^n,g)$ be a Hermitian manifold with $J$ being its almost complex structure and $g=\langle \, , \, \rangle$ denoting its metric, extended bilinearly over ${\mathbb C}$. Let $\nabla$ be the Chern connection. Its torsion and curvature are denoted by $T$ and $R$, respectively:
$$ T(x,y)=\nabla_xy - \nabla_yx - [x,y], \ \ \ R(x,y,z,w) = \langle R_{xy}z,\  w\rangle= \langle \nabla_x\nabla_yz -\nabla_y\nabla_xz -\nabla_{[x,y]}\,z,\  w\rangle $$
for any  tangent vector fields $x$, $y$, $z$, $w$ on $M$. We will extend $T$ and $R$ linearly over ${\mathbb C}$, and still denote them by the same letters. From now on, we will use $X$, $Y$, $Z$, $W$ to denote complex tangent vectors of type $(1,0)$, namely, $X=x-\sqrt{-1}Jx\,$ for some real tangent vector $x$. It is well known that the torsion and curvature of the Chern connection satisfy the following properties:
\begin{equation*}
T(X, \overline{Y})=0, \ \ \ \ \ \ \ R(X,Y,\ast , \ast) = R(\ast , \ast , Z, W)=0.
\end{equation*}
 Therefore the only possibly non-zero components of $R$ are $R(X, \overline{Y}, Z, \overline{W})$, which we will denote by $R_{X \overline{Y} Z \overline{W}}$ for convenience. Let $e=\{ e_1, \ldots , e_n\}$ be a local unitary frame of type $(1,0)$ tangent vectors. Let us denote the components of $T$ by
\begin{equation*}
T(e_i, e_k)= \sum_{j=1}^n  T_{ik}^j e_j.
\end{equation*}
Note that our $T^j_{ik}$ here is equal to twice of the $T^j_{ik}$ in \cite{YZ}.   For any metric connection, the first Bianchi identity (cf. Theorem 5.3 of Ch.\,III of \cite{KN}) takes the form
$$ {\mathfrak S}\{ (\nabla_xT)(y,z) - R_{xy}z - T(x, T(y,z)) \} = 0 ,$$
 where $x$, $y$, $z$ are tangent vectors and ${\mathfrak S}$ means the sum over all cyclic permutation of $x,y,z$. When applied to the special case of Chern connection $\nabla$, we get
\begin{eqnarray}
&& T^{\ell }_{ij,k} + T^{\ell }_{jk,i} + T^{\ell }_{ki,j} \,= \,\sum_r \{ T^r_{ij}T^{\ell}_{kr} + T^r_{jk}T^{\ell}_{ir} +  T^r_{ki}T^{\ell}_{jr} \}, \label{eq:B1} \\
&& R_{k\overline{j}i\overline{\ell}} - R_{i\overline{j}k\overline{\ell}} \ = \   T^{\ell}_{ik,\,\overline{j}}, \label{eq:B1bar}
\end{eqnarray}
for any $1\leq i,j,k,\ell \leq n$, where $e$ is an unitary frame and the index after comma stands for covariant derivative under $\nabla$. See also  \cite[Lemma 7]{YZ} for instance, and notice the change of the factor $2$ in $T^j_{ik}$.

For  any connection  the second Bianchi identity (cf. Theorem 5.3 of Ch.\,III of \cite{KN}) has  the form
 $$ {\mathfrak S}\{ (\nabla_xR)_{yz} + R_{T(x,y)\, z} \} = 0 ,$$
 where ${\mathfrak S}$ means the sum over all cyclic permutation of $x,y,z$, and when applied to the special case of Chern connection $\nabla$, we get
 \begin{equation} \label{eq:B2}
R_{i\overline{j}k\overline{\ell},\, m} - R_{m\overline{j}k\overline{\ell}, \,i} \,= \, \sum_{r=1}^n T^{r}_{im} R_{r\overline{j}k\overline{\ell}},
\end{equation}
 for any $1\leq i,j,k,\ell , m\leq n$, where the indices after comma again stand for covariant derivatives with respect to the Chern connection $\nabla$.

 Now assume that $(M^n,g)$ is CAS, namely, the Chern connection $\nabla$ enjoys $\nabla T=0$, $\nabla R=0$. By (\ref{eq:B1bar}), we know that $R_{k\overline{j}i\overline{\ell}} = R_{i\overline{j}k\overline{\ell}} $
for any indices. That is, the curvature tensor $R$ of $\nabla$ obeys all the symmetry conditions of curvature tensor of a K\"ahler metric.  Recall Gauduchon's torsion $1$-form $\eta$ (see \cite{Gauduchon1}) is defined by (again this is equal to twice of the notion $\eta$ in \cite{YZ}) $\eta = \sum_i \eta_i \varphi_i$, where $\varphi$ is the unitary coframe dual to $e$, and $
\eta_i = \sum_{k=1}^n T^k_{ki}.
$
Some direct calculation shows that
\begin{eqnarray}
\partial \,\omega^{n-1} & = &  - \,\eta \wedge \omega^{n-1}, \label{eq:ba1} \\
\partial \overline{\partial} \,\omega^{n-1} & = & (\overline{\partial} \,\eta + \eta \wedge \overline{\eta})\wedge \omega^{n-1}. \label{eq:ba2}
\end{eqnarray}
Here $\omega$ is the K\"ahler form of $(M^n,g)$. Let $T'=\langle T(X, Y), \bar{Z}\rangle $ be the $(2,1)$-form associated with the torsion tensor $T$. Under a local unitary frame $e$ and dual coframe $\varphi$ we have $T'=\sum_{i,j,k}T^j_{ik}\varphi_i\varphi_k\overline{\varphi}_j = \sum_j \tau_j\overline{\varphi}_j$, and in local holomorphic coordinates we have $T_{ij\bar{k}} dz^i\wedge dz^j\wedge d\bar{z}^k$.  Direct calculation in local coordinates also shows that
\begin{equation}
\partial \omega =\sqrt{-1}\,^t\tau \wedge \overline{\varphi} = \sqrt{-1}T'.
\end{equation}
By \cite[Theorem 3]{YZ}, we know that when $M$ is compact the metric will be balanced (namely, $\eta =0$). It can also be seen as follows. By $\nabla T=0$ we deduce $\bar{\partial}\eta=0$, hence $\eta=0$ by integrating (\ref{eq:ba2}) on $M$. In summary, we have
 \begin{lemma}\label{lemma:1}
 Suppose that $(M^n,g)$ is a Hermitian manifold with a CAS structure. Then it is Chern K\"ahler-like, and under any unitary frame $e$, it holds
  \begin{eqnarray}
  && \sum_{r=1}^n \big( T^r_{ij}T^{\ell}_{kr} + T^r_{jk}T^{\ell}_{ir} +  T^r_{ki}T^{\ell}_{jr} \big) \ = \ 0, \label{eq:TT} \\
 && \sum_{r=1}^n T^{r}_{im} R_{r\overline{j}k\overline{\ell}} \ = \ 0, \label{eq:TR}
\end{eqnarray}
 for any $1\leq i,j,k,\ell , m\leq n$. Moreover, $dT=0$ as a $(2,0)$-form valued in $T^{1, 0}M$,  $\|T\|^2$ is a constant,  and $(M, g)$ is balanced if $M$ is compact. In particular, for compact $M$, $\bar{\partial}^* \omega =0$.
 \end{lemma}


\begin{proof}
For the differential $p$-form $\alpha$ valued in a vector bundle equipped with a metric and a metric compatible connection, we have the formula
\begin{eqnarray}
d\alpha (X_0, \cdots, X_p)&=& \sum_{i=0}^p (-1)^i (\nabla _{X_i} \alpha)(X_0, \cdots, \hat{X}_i, \cdots, X_p)\nonumber\\&\, &-\sum_{0\le i<j\le p} (-1)^{i+j} \alpha (T(X_i, X_j), X_0,\cdots, \hat{X}_i, \cdots, \hat{X}_j, \cdots, X_p). \label{eq:general1}
\end{eqnarray}
Now apply to $T$ we have that for $U, V, Z$ of $(1,0)$-type vectors
$$d\,T(U, V, Z)=T(T(U, V), Z)-T(T(U, Z), V)+T(T(V, Z), U)=0$$
by the fact that $\nabla T=0$ and the Jacobi identity satisfied by $T$. One also have $d\,T(U, V, \overline{Z})=0$ by that $T(U, \overline{Z})=0$.

The last statement follows from the fact that $T$ is parallel.
\end{proof}

 Fix any point $p\in M^n$. Denote the holomorphic tangent space $T'_pM=T^{1,0}_pM\cong {\mathbb C}^n$ at $p$ by $\mathcal{V}_p=\mathcal{V}$. The Chern torsion $T$ is a skew-symmetric bilinear map from $\mathcal{V}\times \mathcal{V}$ to $\mathcal{V}$. Let us denote by $\mathcal{W}$ the subspace of $\mathcal{V}$ spanned by the image set of $T$:
 $$ \{ T(X,Y) \mid X, Y \in \mathcal{V}\}.$$
 Denote by
 $$ \mathcal{N}=\{ Z \in \mathcal{V} \mid \langle T(X,Y), \overline{Z} \rangle =0 \ \ \forall \ X, Y \in \mathcal{V}\}$$
 the linear subspace of $\mathcal{V}$, which is the set of vectors in $\mathcal{V}$ that is perpendicular to $\mathcal{W}$.  We  call $\mathcal{N}$ the {\em perpendicular space} of $T$. The  tangent bundle $ T^{1,0}M $ splits orthogonally as $\mathcal{W} \oplus \mathcal{N}  $. We will call $\mathcal{W}$ the {\em image distribution} of $T$, since it is easy to see that $\dim(\mathcal{W})$ is independent of the choice of point $p$. Moreover we have the following

\begin{lemma}\label{lemma:32-1}
 Let $(M^n,g)$ be a Hermitian manifold with a CAS structure. Then the image distribution $\mathcal{W}$ of the  torsion $T$ is invariant under the parallel transport.  In particular, $\nabla \mathcal{ W} \subseteq \mathcal{W}$.
 \end{lemma}
 \begin{proof} If $X(t), Y(t)$ are parallel along a curve $\gamma(t)$, $T(X(t), Y(t))$ is parallel. This proves the first statement. Let $Z$ be a real vector at $p$ with $\gamma(t)$ be a curve satisfying $\gamma(0)=p$, $\gamma'(0)=Z$. Let $X(t), Y(t)$ be two $(1, 0)$-type vectors along $\gamma(t)$. Let $W(t)=T(X(t), Y(t))$. Let $P^{t, 0}$ be the parallel transport from $\gamma(t)$ to $\gamma(0)$ along $\gamma$. By Lemma 2.1 of \cite{Ni-21} $\nabla_Z W(0)=\lim_{t\to 0}\frac{1}{t}\left(P^{t, 0}(T(X(t), Y(t)))-T(X(0), Y(0))\right)$. Since $T$ is parallel we have that
 $$
 \nabla_Z W(0)=\lim_{t\to 0}\frac{1}{t}\left(T(P^{t, 0}(X(t)), P^{t, 0}(Y(t)))-T(X(0), Y(0))\right)
 $$
 which clearly belongs to $\mathcal{W}$.
 \end{proof}

As an immediate corollary, we know that $\mathcal{W}$ is a holomorphic foliation in $M$:

\begin{lemma}\label{lemma:3}
 Suppose that $(M^n\!,g)$ is a Hermitian manifold with a CAS structure. Then the image distribution $\mathcal{W}\,$ of $T$ is a flat holomorphic foliation in $M$.
 \end{lemma}

 \begin{proof}
 $\mathcal{W}$ being parallel under $\nabla$ implies that it is holomorphic, so it suffices to prove that $\mathcal{W}$ is integrable. For any type $(1,0)$ vector fields $X$, $Y$ in $\mathcal{W}$, since $T(X,Y)\in \mathcal{W}$ by the construction of $\mathcal{W}$, we have
 $$ [X,Y] = \nabla_XY - \nabla_YX - T(X,Y) \in \mathcal{W}.$$
  So $\mathcal{W}$ is integrable. By (\ref{eq:TR}) and the definition of $\mathcal{W}$, we know that $\mathcal{W}$ is contained in the kernel of the curvature tensor $R$, hence flat.
 \end{proof}

 Since $\mathcal{W}$ is parallel, its orthogonal complement $\mathcal{N}=\mathcal{W}^{\perp}$ is also parallel. However, $\mathcal{N}$ may not be integrable in general. This is because for any type $(1,0)$ vector fields $X$, $Y$ in $\mathcal{N}$, we always have $[X,\overline{Y}]\in \mathcal{N}\oplus \overline{\mathcal{N}}$ as before (since $T(X, \overline{Y})=0$), but $[X,Y]$ may not be in $\mathcal{N}$ since $T(X,Y)$ may be non-zero (hence lives in $\mathcal{W}$). If $\mathcal{N}$ is (complex) $1$-dimensional, then this will not be an issue, so in this case $\mathcal{N}$ is also integrable, hence $M$ is locally a (possibly a warped) product as a Hermitian manifold. We have the following

 \begin{theorem}\label{thm-imagecodim1}
 Suppose that $(M^n,g)$ is a Hermitian manifold with a CAS structure such that the image distribution $\mathcal{W}\,$ of $T$ has codimension one. If $g$ is not Chern flat, then the universal cover of $M$ is holomorphically isometric to a product $G\times C$ where $G$ is a connected, simply-connected complex Lie group equipped with a left invariant Hermitian metric, and $C$ is $\mathbb{P}^1$ or the unit disc $\mathbb{D}=\{z\, |\, |z|<1\}$ equipped with (a scaling of) the standard metric.
 \end{theorem}

\begin{proof}
In this case, the holomorphic tangent bundle $T^{1,0}M$ splits orthogonally as the direct sum of two integrable holomorphic foliation $\mathcal{W}\oplus \mathcal{N}$. Since $\mathcal{W}$ is contained in the kernel of the Chern curvature tensor $R$, around any given point $p\in M$ we have local unitary frame $e$ so that $\{ e_2, \ldots , e_n\}$ spans $\mathcal{W}$ and $\nabla e_2=\cdots =\nabla e_n=0$.

We claim that $T^{\ast}_{1\ast}=0$. Assume the contrary, then there would be some $2\leq \alpha, \beta \leq n$ such that $T^{\alpha}_{1\beta}\neq 0$. Consider the local $(1,0)$-form $\psi$ defined by $\psi (\cdot ) = \langle T(\cdot , e_{\beta} ) , \overline{e}_{\alpha} \rangle$. Let $X$ be the vector field dual to $\psi$, namely, $\psi (Y)= \langle Y, \overline{X}\rangle$ for any $Y$. Then $X=\sum_{i=1}^n \overline{T^{\alpha}_{i\beta}}\, e_i = X_1 + X_2$ where $X_1= \overline{T^{\alpha}_{1\beta}}\, e_1$ is the component of $X$ in $\mathcal{N}$. Clearly, $\nabla \psi=0$, hence $\nabla X=0$. So its component in the parallel distribution $\mathcal{N}$ is also parallel, namely, $\nabla X_1=0$. Replacing $e_1$ by $\frac{1}{|X_1|}X_1$, we get $\nabla e_1=0$ thus $e$ is a parallel frame. This means that $M$ is Chern flat, a contradiction. This completes the proof of the claim.

Denote by $\nabla^{LC}$ the Levi-Civita connection of $g$. By \cite[Lemma 2]{YZ}, note that our notation $T^j_{ik}$ here is twice of the torsion defined there, we have
$$ \nabla^{LC}_{e_i}e_1 = \nabla_{e_i}e_1 + \frac{1}{2} \sum_{j=1}^n  T^j_{1i}e_j, \ \ \ \  \nabla^{LC}_{\overline{e}_i}e_1 = \nabla_{\overline{e}_i}e_1  +  \frac{1}{2} \sum_{j=1}^n \{  T^i_{1j}\overline{e}_j  - \overline{ T^1_{ji} } e_j\} .$$
So by the above claim,  we know that $\mathcal{N}\oplus \overline{\mathcal{N}}$ is parallel under the Levi-Civita connection $\nabla^{LC}$, hence splits off a one-dimensional K\"ahler de Rham factor $C$. This completes the proof Theorem \ref{thm-imagecodim1}.
\end{proof}

Following a similar argument,  we can split off some other K\"ahler de Rham factors of a given CAS manifold. Let
$$\mathcal{N}_0\doteqdot \{X \in \mathcal{N}\ |\   T(X, Y)=0, \ \ \ \forall \ Y\in \mathcal{V} \}.$$
We will call it the {\em kernel space} of $T$.  It is a linear subspace of $\mathcal{V}$, and we claim that it is invariant under the parallel transport:

 \begin{lemma}\label{lemma:4}
 Let $(M^n,g)$ be a Hermitian manifold with a CAS structure. Then the kernel distribution $\mathcal{N}_0$ of $T$ is parallel under $\nabla$.
 \end{lemma}
\begin{proof} The argument in the proof Lemma \ref{lemma:3} also implies the result. We include an alternate argument using a moving frame for the sake of later proofs. Denote by $\mathcal{N}_1$ the orthogonal complement of $\mathcal{N}_0$ in $\mathcal{N}$. Let $e$ be a local unitary frame such that $\{ e_1, \ldots , e_q\}$ spans $\mathcal{W}$,  $\{ e_{q+1}, \ldots , e_{p}\}$ spans $\mathcal{N}_1$, and $\{ e_{p+1}, \ldots , e_{n}\}$ spans $\mathcal{N}_0$. Let $\theta$ be the connection matrix with respect to the frame $e$. Namely $(\theta_{ab})$  matrix (of 1-forms) is defined by $\nabla e_a =\theta_{ab} e_b$.
  For any $1\leq a,b,c\leq n$ and any $x\in \mathcal{V} \oplus \overline{\mathcal{V}}$, since $\nabla T=0$, we have $T^{c }_{ab,x} =0$, that is
 \begin{equation}
  x(T^{c}_{ab}) = \sum_{r=1}^n \{ T^c_{rb} \theta_{ar}(x) +  T^c_{ar} \theta_{br}(x) - T^r_{ab} \theta_{rc}(x) \} . \label{eq:nablaT}
  \end{equation}
  We want to show that $\nabla \mathcal{N}_0 \subseteq \mathcal{N}_0$, namely, $\theta_{i\alpha}=0$ for any $p<\alpha \leq n$ and any $1\leq i\leq p$. We already know this for any $1\leq i\leq q$, so we may assume that $q<i\leq p$.

 Let us denote by $\mathcal{N}^{\ast}$ the dual space of $\mathcal{N}$, with the induced inner product. For $X$, $Y\in \mathcal{V}$, let us denote by $\ell_{X\overline{Y}} \in \mathcal{N}^{\ast}$ the linear functional on $\mathcal{N}$ defined by:
 $$ \ell_{X\overline{Y}}(Z) = \langle T(Z,X), \overline{Y}\rangle .$$
 Let $\Delta \subseteq \mathcal{N}^{\ast}$ be subset consisting of $\ell_{X\overline{Y}}$ for all  $X,Y\in \mathcal{V}$, and let $U={\mathbb C}(\Delta)$ be its linear span in $\mathcal{N}^{\ast}$. By definition, $\mathcal{N}_0$ consists of elements $Z\in \mathcal{N}$ such that $\ell_{X\overline{Y}}(Z)=0$ for any $X,Y\in \mathcal{V}$, namely, $\mathcal{N}_0\subseteq \mathcal{N}$ is the subspace annihilated by $\Delta$ in $\mathcal{N}^{\ast}$. Therefore $\mathcal{N}_0=Ann(\Delta ) = Ann (U)$, and the decomposition $\mathcal{N}=\mathcal{N}_1 \oplus \mathcal{N}_0$ corresponds to $\mathcal{N}^{\ast} = U \oplus U^{\perp}$. In particular, there exists a basis $\{ \varphi_i, \varphi_{\alpha}\}$ of $\mathcal{N}^{\ast}$ which is dual to $\{ e_i, e_{\alpha}\}$ of $\mathcal{N}$. Here $q<i\leq p$ and $p<\alpha \leq n$.

 Now let us fix $\alpha $ with $p<\alpha \leq n$. Since $T^{\ast}_{\alpha \ast }=0$, by letting $a=\alpha$ in (\ref{eq:nablaT}), we get
 $$ \sum_{r=1}^n T^c_{r b} \theta_{\alpha r} =0 \ \ \ \ \ \forall \ 1\leq b, c\leq n. $$
 In particular, $\sum_r T^Y_{rX}\theta_{\alpha r}=0$ for any $X,Y\in \mathcal{V}$.

 For any fixed $i$ with $q<i\leq p$, since $\varphi_i \in U ={\mathbb C}(\Delta)$, there exist finitely many elements $X_1$, $Y_1, \ldots , X_k$, $Y_k$ in $\mathcal{V}$ such that
 $ \varphi_i = \ell_{X_1\overline{Y_1}} + \cdots + \ell_{X_k\overline{Y_k}}$. So we get
 $$ \theta_{\alpha i} = \sum_{r=q+1}^n \varphi_i(e_r) \theta_{\alpha r} = \sum_r \big( T^{Y_1}_{rX_1} + \cdots +  T^{Y_k}_{rX_k} \big) \theta_{\alpha r} = 0. $$
 This completes the proof of the lemma.
 \end{proof}

Let us denote by $\nabla^{LC}$ the Levi-Civita connection of $g$. Suppose $X$ is any type $(1,0)$ vector field belonging to $\mathcal{N}_0$. Then by definition, we have $T(X,Y)=0$ and $\langle T(Y,Z),\overline{X}\rangle =0$ for any type $(1,0)$ vectors $Y$ and $Z$. Since the difference between $\nabla^{LC}$ and the Chern connection $\nabla$ are given by the torsion components \cite[Lemma 2]{YZ} (see also (\ref{eq:ts})), so as in the proof of Theorem \ref{thm-imagecodim1} we have $\nabla^{LC} X = \nabla X$, therefore $\mathcal{N}_0\oplus \overline{\mathcal{N}_0}$ is parallel under $\nabla^{LC}$, hence giving a de Rham factor which is K\"ahler since its Chern torsion vanishes. In summary we have

\begin{theorem}\label{thm-N0}
 Let $(M^n,g)$ be a Hermitian manifold with a CAS structure. Suppose that the kernel foliation $\mathcal{N}_0$ of $T$ is of dimension $s$,  then the universal cover of $M$ is holomorphically isometric to a product $M_1\times M_2$, where $M_2$ is a Hermitian symmetric space of dimension $s$, while the Chern connection of $M_1$ has parallel torsion and curvature, and its torsion has trivial kernel.
\end{theorem}

\section{Proof of Theorem \ref{thm-dim3}  and \ref{thm-dim4}}

In this section, we will prove Theorem \ref{thm-dim3}  and \ref{thm-dim4}, which give  a classification of CAS manifolds in dimensions $3$ and $4$. For readers benefit, let us restate Theorem \ref{thm-dim3} below:

\begin{theorem}\label{thm-dim3b}
Let $(M^3,g)$ be a compact  CAS manifold, namely a compact Hermitian manifold whose Chern connection has parallel torsion and curvature. Then it is either K\"ahler (hence a locally Hermitian symmetric space) or Chern flat.
\end{theorem}

\begin{proof}  Let us now specialize to dimension $3$ and let $(M^3,g)$ be a CAS manifold. We assume that $g$ is not K\"ahler. Hence the image foliation $\mathcal{W}$ is non-zero. Since $\mathcal{W}$ is contained in the kernel of $R$ by (\ref{eq:TR}), we know that $g$ would be Chern flat if  $\mathcal{W}$ is $3$-dimensional. Also, by Theorem 5, if $\mathcal{W}$ is $2$-dimensional, then $M$ will split off a $1$-dimensional K\"ahler de Rham factor, while the other factor, being $2$-dimensional and balanced, is also K\"ahler. This will imply that $g$ is K\"ahler, a contradiction. So we are only left with the possibility of $\mathcal{W}$ being $1$-dimensional.

Let $e$ be any local unitary frame such that $e_3\in \mathcal{W}$. We will call such a frame {\em admissible}. By (\ref{eq:TR}), we have $R_{3\overline{\ast} \ast \overline{\ast}} =0$. First let us examine the components of $T$ under an admissible frame. We have $T^1_{\ast \ast } = T^2_{ \ast \ast }=0$ by the definition of $\mathcal{W}$. Also, since $\eta=0$, we have
$$ T^3_{31} = \eta_1 - T^2_{21} = \eta_1 =0.$$
Similarly, $T^3_{32}=0$, thus the only possibly non-zero component of $T$ would be $T^3_{12}$. Since $\nabla T=0$, the norm $|T|^2$ must be a constant, and this constant is non-zero since $g$ is not K\"ahler. By scaling the metric $g$ if necessary, we will assume that  $|T^3_{12}|=1$ from now on.

By the argument in the proof of Lemma \ref{lemma:4}, the connection matrix $\theta$ under an admissible frame $e$ satisfy $\theta_{13}=\theta_{23}=0$. Also, we have that the curvature $\Theta$ satisfies that $\Theta_{33}=d\theta_{33}=0$.  By $\nabla T=0$, we have $ \nabla (T(e_1,e_2)) = T(\nabla e_1, e_2) + T(e_1, \nabla e_2)$, which leads to
\begin{equation}
 d \log T^3_{12} =  \theta_{11} + \theta_{22} - \theta_{33} \label{eq:logT}
\end{equation}
Taking differential in the above identity, and use the fact $d\theta_{33}=0$, we get
\begin{equation}
 tr (\Theta )  = d \,tr (\theta) = d( \theta_{11} + \theta_{22} + \theta_{33}) =  d(  \theta_{11} + \theta_{22} - \theta_{33}) = 0. \label{eq:Ricci}
\end{equation}
That is, {\em the Chern curvature of $g$ is Ricci flat}. This is a crucial point for us, and we will use it along with the fact $\nabla R=0$ to force $R=0$, thus completing the proof of Theorem \ref{thm-dim3}. (In case that $M$ is complex homogenous, and admits an invariant volume form, applying  Theorem A and Proposition 5.1 of \cite{HK}, the compactness of $M$ and the flatness of Ricci implies that $M$  is parallelizable, which is Chern flat with respect to the induced left invariant metric. The analysis below shows directly  the result for the case when $M$ is CAS only.)

Let $e$ be an admissible frame. Since $\Theta_{33}= d\theta_{33}=0$, we may rotate $e_3$ and assume that $\nabla e_3=0$. Fix this $e_3$, then we know that $\theta_{3i}=\theta_{i3}=0$ for any $1\leq i\leq 3$. The curvature components $R_{i\overline{j}k\overline{\ell}}$ are zero if any index is $3$. Since $R$ obeys all the K\"ahler symmetries, we only have the following components:
$$ A=R_{1\overline{1}1\overline{1}}, \ \ B=R_{1\overline{1}2\overline{2}}, \ \ C=R_{2\overline{2}2\overline{2}}, \ \ D=R_{1\overline{2}1\overline{2}}, \ \ E=R_{1\overline{1}1\overline{2}}, \ \ F=R_{2\overline{2}1\overline{2}}.$$
Since $tr(\Theta )=0$, we always have $A+B=C+B=E+F=0$. At any fixed point $p\in M$, by a unitary change of $\{ e_1, e_2\}$, we may assume that $A$ equals the maximum value of holomorphic sectional curvature at $p$. This implies that $E=0$. Replace $e_1$ by an appropriate $\rho e_1$ where $|\rho |=1$, we may further assume that $D=|D|$. This kind of choice can be made at every point in $M$, and clearly we can choose local unitary frame $e$ so that $e_3$ is parallel and under $\{ e_1, e_2\}$ the curvature components satisfy the above requirements. In this case, we have
$$ A=C=-B=H, \ \ D=|D| \geq 0, \ \ E=F=0. $$
It is easy to see that $3H\geq |D|$, and the maximum  (the minimum) of the holomorphic sectional curvature at the point $p$ is equal to $H$ ($-\frac{1}{2}(H+|D|)$, respectively). By  $\nabla R=0$, we know that both  $H$ and $|D|$ are constants.  Our goal here is to show that the metric $g$ has vanishing Chern curvature, or equivalently, $H=D=0$.  The assumption $\nabla R=0$ yields
\begin{equation*}
d(R_{i\overline{j}k\overline{\ell}}) = \sum_r \big( \theta_{ir}R_{r\overline{j}k\overline{\ell}} - \overline{\theta_{rj}}R_{i\overline{r}k\overline{\ell}} + \theta_{kr}R_{i\overline{j}r\overline{\ell}} - \overline{\theta_{r\ell}}R_{i\overline{j}k\overline{r}} \big).
\end{equation*}
By letting $(ijk\ell) = (1212)$ and $(1112)$, respectively, we get
\begin{equation}  0 = 2(\theta_{11}-\theta_{22})D, \ \ \ \ \ 0=\theta_{12}(2B-A)- \theta_{21}D = -3H\theta_{12} -D \theta_{21}.
\label{eq:nablaR}
\end{equation}

If $D>0$, then $\theta_{11}=\theta_{22}$, and $\theta_{21} = -\frac{3H}{D} \theta_{12}$, so
$$ \Theta_{11} = d\theta_{11} - \theta_{12} \wedge \theta_{21} = d\theta_{11}.$$
Similarly, $\Theta_{22}=d\theta_{22}$, so $\Theta_{11}=\Theta_{22}$. Since $\Theta_{11}+\Theta_{22}=0$, we get $\Theta_{11}=0$. Hence $H=0$. By $3H\geq |D|$, we get $D=0$, contradicting to the hypothesis $D>0$. This means that we must have $D=0$.

If $H>0$, then the second equality in (\ref{eq:nablaR}) yields $\theta_{12}=0$. So
$$ \Theta_{12}=d\theta_{12} - \theta_{11}\wedge \theta_{12} - \theta_{12} \wedge \theta_{22} = 0.$$
On the other hand,
$$\Theta_{12} = D\varphi_1\wedge \overline{\varphi_2} + B\varphi_2\wedge \overline{\varphi_1} = - H \varphi_2\wedge \overline{\varphi_1} \neq 0,$$
a contradiction. So we must have $H=0$ as well. The vanishing of both $H$ and $D$ means that the metric $g$ is Chern flat, and we have  completed the proof of Theorem \ref{thm-dim3b}, which is Theorem \ref{thm-dim3}. \end{proof}

Again for readers benefit, let us restate Theorem \ref{thm-dim4} below:

\begin{theorem}\label{thm-dim4b}
Let $(M^4,g)$ be a CAS manifold, namely a compact Hermitian manifold whose Chern connection has parallel torsion and curvature. Then it is either K\"ahler (hence a locally Hermitian symmetric space), or Chern flat, or the universal cover is holomorphically isometric to the product $C\times G$, where $C$ is a complex space form of dimension $1$ and $G$ a complex Lie group of dimension $3$ (with left invariant metric).
\end{theorem}

\begin{proof} Let $(M^4,g)$ be a CAS manifold. First, we will assume that $g$ is not K\"ahler, and by Theorem \ref{thm-N0}, we will assume that the kernel $\mathcal{N}_0$ of the Chern torsion tensor $T$ is trivial. Furthermore, by Theorem \ref{thm-imagecodim1}, we may assume that the image distribution $\mathcal{W}$ of $T$ has dimension $1$ or $2$.

If $\mathcal{W}$ has dimension $1$, then we claim that $\mathcal{N}_0$ will be non-zero. To see this, let $e$ be a local unitary frame such that $\mathcal{W}$ is spanned  by $e_4$. That is, $T^i_{\ast \ast}=0$ for $1\leq i\leq 3$. This implies  that $T^4_{4i}=0$ for any $i$ as $\eta =0$. So $(T^4_{ij})$ becomes a skew-symmetric matrix on $\mathcal{N}=\mathcal{W}^{\perp}\cong {\mathbb C}^3$. It must have a non-trivial kernel space, which is contained in $\mathcal{N}_0$, contradicting to the assumption that $\mathcal{N}_0=0$.

We are left with the case that $\mathcal{W}$ has dimension $2$. Let us choose a local unitary frame $e$ so that $\mathcal{W}$ is spanned by $\{ e_3, e_4\}$. We may assume that $\nabla e_3=\nabla e_4=0$ since $\mathcal{W}$ is contained in the kernel of $R$. We will call such a local frame {\em admissible}. Note that we have $T^{3}_{34}=T^4_{34}=0$ by $\eta =0$.

Let $a=|T^3_{12}|$ and $b=|T^4_{12}|$. Since  $T^3_{12}=\langle T(e_1,e_2), \overline{e}_3\rangle$, which depends only on $e_1\wedge e_2$, so its absolute value $a$ is independent of the choice of an admissible frame, thus is a constant. Similarly, $b$ is also a constant. We note that the distribution $\mathcal{N}$, as the orthogonal complement of $\mathcal{W}$, is parallel under $\nabla$.

If both $a$ and $b$ are zero, then $T(e_1,e_2)=0$, (and of course $T(e_i, \overline{e}_j)=0$) hence $\mathcal{N}$ is integrable. In this case, $M$ is locally the product of two complex surfaces, with perpendicular leaves, and by (\ref{eq:ts}) we have $S_{e_1} e_2=S_{e_2}e_1=0=S_{e_3} e_4=S_{e_4} e_3$, thus locally $(M, g)$ is a product  Hermitian manifold. This will force the torsion tensor to split, in fact $T=0$ hence a contradiction. So either $a$ or $b$ must be non-zero. By a constant unitary change of $\{ e_3, e_4\}$ if necessary, we may assume that $T^4_{12}$ vanishes at a point, hence we may assume that $b=0$ and $a>0$ from now on.

Denote by $\varphi$ the coframe dual to $e$, and consider the local $(1,0)$-form $\psi$ given by $\psi = T^3_{13} \varphi_1 + T^3_{23}\varphi_2$. By the fact that $T^3_{34}=0$, we know that  $\psi = \langle T( \cdot , e_3), \overline{e}_3\rangle$. Therefore, $\psi$ is parallel under  $\nabla$ as both $T$ and $e_3$ are. If $\psi \neq 0$, then $\mathcal{N}$ contains a parallel local section (given by the dual vector of $\psi$), hence by a unitary change of $\{ e_1, e_2\}$ we may assume that $\nabla e_1 =\nabla e_2 =0$, leading to the conclusion that $g$ is Chern flat. So we may always assume that $\psi =0$, or equivalently, $T^3_{\ast 3}=0$, and by the same argument, $T^4_{\ast 4}=0$.

More generally for any given (local) parallel holomorphic vectors $Z, W$,  define $\psi(\cdot)=\langle T(\cdot, W), \overline{Z}\rangle$. We have
\begin{lemma}\label{lemma:41}
The $(1, 0)$-form $\psi$ defined as above $\psi$ is a parallel holomorphic 1-form.
\end{lemma}
\begin{proof} For any local complex tangent vectors $U$, $V$ with $V$ being of type $(1,0)$, we have
$$(\nabla_U\psi)(V)=U(\psi(V))-\psi(\nabla_U V)=U\langle T(V, W), \overline{Z}\rangle -\langle T(\nabla_U V, W), \overline{Z}\rangle=0.$$
So $\psi$ is parallel and holomorphic.
\end{proof}

Applying the same argument as in the paragraph before the lemma, we know that $T^3_{i4} = T^4_{i3}=0$ hold for any $i$. In particular, $T^4_{\ast \ast }=0$, contradicting to the fact that $e_4\in \mathcal{W}$. This shows that when $\mathcal{N}_0=0$ and $\dim \mathcal{W} =2$, the manifold has to be Chern flat, thus we have completed the proof of Theorem \ref{thm-dim4b}, which is Theorem \ref{thm-dim4}. \end{proof}

\section{A discussion in dimension 5}

Having seen the restrictive nature of CAS manifolds in dimensions $3$ and $4$, one naturally wonder if the phenomenon will persist in higher dimensions. In this section, we give a brief discussion in dimension $5$, where the previous line of argument runs into a problem.

To simplify things, we will call a CAS manifold $(M^n, g)$ {\em trivial} if all of its de Rham factors are either K\"ahler (hence Hermitian symmetric) or Chern flat (i.e., complex Lie groups). In this terminology, Theorem \ref{thm-dim3} and \ref{thm-dim4} simply say that all CAS manifolds are trivial when $n\leq 4$. For $n\geq 5$, the question is whether there are non-trivial examples, and if there are, give characterization or even classification of them.

Now let us assume that $(M^5,g)$ is a CAS manifold. We will assume that $g$ is not K\"ahler (i.e., $T\neq 0$), and the kernel $\mathcal{N}_0$ of $T$ is trivial (i.e., $M^5$ does not have a K\"ahler de Rham factor). By Theorem \ref{thm-imagecodim1}, we may also restrict ourselves to the $1\leq \dim \mathcal{W}\leq 3$ cases, where $\mathcal{W}$ is  the image distribution  of $T$.

Since $\mathcal{W}$ is parallel under $\nabla$ and is contained in the kernel of $R$, around any given point $p\in M$, there always exists a local unitary frame $e=\{e_1, \ldots , e_5\}$ so that the last $\dim \mathcal{W}$ terms of $e$ are all parallel and they span $\mathcal{W}$. We will call such a frame {\em admissible}.

\vspace{0.3cm}

{\bf Case 1: $\dim \mathcal{W}=3$.}

Let $e=\{ e_1, \ldots , e_5\}$ be a local admissible frame, where $\nabla e_3=\nabla e_4=\nabla e_5=0$ and $\{ e_3, e_4, e_5\}$ spans $\mathcal{W}$. The $(1,0)$-form $\psi =T^3_{13}\varphi_1 + T^3_{23}\varphi_2$ is the difference between $\psi_1 =\langle T(\cdot , e_3), \overline{e}_3\rangle$ and $\psi_2= T^3_{43}\varphi_4 + T^3_{53}\varphi_5$. Since $T$, $e_3$, $\varphi_4$, $\varphi_5$ are all parallel, we know that $\psi_1$ and $\psi_2$, hence $\psi$, is parallel under $\nabla$. If $\psi \neq 0$, then it will give us a parallel direction in $\mathcal{N}=\mathcal{W}^{\perp}$, hence we may choose $e$ so that $\nabla e_1=\nabla e_2=0$, this yields that $M^5$ is Chern flat. Therefore, when $M^5$ is not Chern flat, we must have $T^3_{13}=T^3_{23}=0$. Replacing $e_3$ by $e_3+ae_4+be_5$ for any fixed constants $a$ and $b$, the same argument implies that
\begin{equation}
 T^j_{i1}=T^j_{i2}=0, \ \ \ \ \forall \ 3\leq i,j\leq 5. \label{eq:Tij=0}
 \end{equation}
This implies that there must be $i$ so that $T^i_{12}\neq 0$, as otherwise we will have $e_1\in \mathcal{N}_0$, a  contradiction. By a constant unitary change of $\{ e_3, e_4, e_5\}$ if necessary, we may assume that $T^3_{12}\neq 0$ while $T^4_{12}=T^5_{12}=0$. Note that  $a=|T^3_{12}|$ is a constant. By (\ref{eq:TT}), we have
 $$ \sum_{r=1}^5 \big( T^r_{ij}T^{\ell}_{kr} + T^r_{jk}T^{\ell}_{ir}+ T^r_{ki}T^{\ell}_{jr} \big)  =0.$$
Let $i=1$, $j=2$, and $k, \ell \geq 3$, then the second and third term in the parenthesis vanish by (\ref{eq:Tij=0}), hence
$$ T^3_{12}T^{\ell}_{k3} =0 \ \Longrightarrow \ T^{\ell}_{k3} =0 \ \ \ \  \forall \ 3\leq k,\ell \leq 5. $$
So the only possibly non-zero components of $T$ are $T^3_{12}\neq 0$ and $T^{\ast}_{45}$. By $\eta =0$, we get $T^4_{45} = - T^3_{35}=0$, $T^5_{45}= - T^3_{43} =0$. So the only possibly non-zero components of $T$ are $T^3_{12}$ and $T^3_{45}$. In particular, $T^4_{\ast \ast }=T^5_{\ast \ast }=0$, contradicting with the fact that $e_4, e_5 \in \mathcal{W}$. So in the $\dim \mathcal{W} =3$ case the manifold  must be Chern flat.

\vspace{0.3cm}

 {\bf Case 2: $\dim \mathcal{W}=2$.}

Let $e=\{ e_1, \ldots , e_5\}$ be a local admissible frame, where $\nabla e_4=\nabla e_5=0$ and $\{ e_4, e_5\}$ spans $\mathcal{W}$. Since $\eta =0$, we get $T^4_{45}=T^5_{45}=0$. For any non-trivial constant linear combination $X=ae_4+be_5$, since $\nabla X=0$, we know that  $\psi_X = \langle T( \cdot , X), \overline{X}\rangle $ is parallel, and $\psi_X = \sum_{r=1}^3 T^X_{rX} \varphi_r $. If $\psi_X \neq 0$, then it corresponds to a parallel direction in $\mathcal{N}$. Note that if $\mathcal{N}$ contains two linearly independent parallel directions, then it will admit a parallel frame, hence $M^5$ will be Chern flat. Assume that $M$ is not Chern flat, then  either (1): $\psi_X=0$ for all $X$, or (2): some $\psi_X\neq 0$ while all other $\psi_Y$ is proportional to this $\psi_X$.

{\em Subcase (1).} In this case we have $T^{\alpha }_{i\beta}=0$ for any $1\leq i\leq 3$ and any $4\leq \alpha , \beta \leq 5$. Restrict $T^4$ on $\mathcal{N}\cong {\mathbb C}^3$. It is a parallel skew-symmetric form on $\mathcal{N}$. Denote by $V$  its kernel space. If $V=\mathcal{N}$, then $T^4_{ij}=0$ for all $1\leq i,j\leq 3$ hence $T^4_{\ast \ast }=0$, contradicting with $e_4\in \mathcal{W}$. So $V$ must be one-dimensional, leading to a parallel direction in $\mathcal{N}$. Similarly, restricting $T^5$ on $\mathcal{N}$, and its one-dimensional kernel $V'$ also gives a parallel direction in $\mathcal{N}$. If $V\neq V'$, then $\mathcal{N}$ contains two parallel directions, hence it admits a parallel frame. This implies that $M$ is Chern flat. So we may assume that $V=V'$. Without loss of generality, assume $e_1\in V=V'$. Then we have $T^{\alpha}_{1\ast }=0$, so $e_1\in \mathcal{N}_0$ by definition, which contradicts to our assumption. This completes the discussion of {\it Subcase (1)}.

{\em Subcase (2).} In this case, without loss of generality, we may assume that $\psi_{e_4}\neq 0$ and corresponds to $e_3\in \mathcal{N}$. So we have $\nabla e_3=0$, and $T^4_{34}\neq 0$, while $T^4_{14}=T^4_{24}=0$. Also, since for any constant $\lambda$, $\psi_X$ is always proportional to $\psi_{e_4}$ where $X=e_4+\lambda e_5$. This means that $T^X_{1X}=T^X_{2X}=0$. By the arbitrariness of $\lambda$, we get $T_{1\alpha }^{\beta}=T_{2\alpha }^{\beta}=0$ for any $4\leq \alpha , \beta \leq 5$.

Let us again restrict $T^4$ on $\mathcal{N}\cong {\mathbb C}^3$. It is a parallel skew-symmetric bilinear form on $\mathcal{N}$. So its kernel space $V\subseteq \mathcal{N}$ is either one-dimensional or $V=\mathcal{N}$. If $V=\mathcal{N}$, we get $T^4_{ij}=0$ for $1\leq i,j\leq 3$. If $\dim V=1$, then since $V$ is parallel, if it does not contain $e_3$, then $\mathcal{N}$ will have two parallel directions, hence admits a parallel frame, which leads to $M$ being Chern flat. So we must have $e_3\in  V$. Either way, we will have $T^4_{3i}=0$ for $i=1,2$. By the same token, $T^5_{3i}=0$, $i=1,2$. So the only possibly non-zero components of $T$ are $T^{\alpha}_{3\beta}$ and $T^{\alpha}_{12}$, where $4\leq \alpha , \beta \leq 5$, and $e_3$, $e_4$, $e_5$  are parallel.

Note that now we are in the situation of the proof of Theorem \ref{thm-dim3}. The only possibly non-zero components of $R$ are $R_{i\overline{j}k\overline{\ell}}$ with $1\leq i,j,k,\ell \leq 2$. The curvature tensor obeys all the K\"ahler symmetries. We claim that it must be Ricci flat.

Note that if both $T^4_{12}$ and $T^5_{12}$ were zero, then $e_1$ and $e_2$ would belong to $\mathcal{N}_0$, a contradiction. Without loss of generality, let us assume that $T^4_{12}\neq 0$. This value depends only on $e_1\wedge e_2$, so its absolute value is a constant. Now if we replace $e_i$ by $\rho_ie_i$, $i=1,2$, where $\rho_i$ are smooth functions with $|\rho_i|=1$, we may assume that $T^4_{12}=|T^4_{12}|$ is a constant. Since $\nabla T=0$, we get
$$ 0 = d(T^4_{12}) = \sum_{r=1}^5 \big( T^4_{r2}\theta_{1r} + T^4_{1r}\theta_{2r} - T^r_{12}\theta_{r4}\big) = T^4_{12}(\theta_{11}+\theta_{22}).$$
This implies that $\theta_{11}+\theta_{22}=0$, hence $\mbox{tr}(\Theta)=0$, and the claim is proved.

Next we follow the argument in the proof of Theorem \ref{thm-dim3}. By choosing $e_1$ to be the direction that realizes the maximum value of the holomorphic sectional curvature $H$ at each point, we have $R_{1\overline{1}1\overline{2}}=0$, hence $R_{2\overline{2}1\overline{2}}=0$ by the Ricci-flatness. While $R_{1\overline{1}1\overline{1}}=R_{2\overline{2}2\overline{2}}=-R_{1\overline{1}2\overline{2}}=a$ is a constant, and $|D|$ is also a constant where $D=R_{1\overline{2}1\overline{2}}$. Since Ricci-flatness implies that the scalar curvature is zero, and the scalar curvature is (a positive multiple of) the average value of $H$ by Berger's formula, so if $a\leq 0$ then $H\equiv 0$ and $M$ is Chern flat. Therefore, we may assume that $a>0$ from now on.

If we replace $e_i$ by $\rho_ie_i$, $i=1,2$, where $\rho_i$ are smooth functions with $|\rho_i|=1$. Then $T^4_{12}$ is changed to $\rho_1\rho_2T^4_{12}$, while $D$ is changed to $\rho_1^2 \overline{\rho_2^2}D$. Therefore, we can choose $e_1$ and $e_2$ so that $T^4_{12}=|T^4_{12}|>0$ and $D=|D|\geq 0$ at the same time. So we have $\theta_{11}+\theta_{22}=0$. Since $\nabla R=0$, we get
$$ 0 = dD = 2 \sum_{r=1}^5 \big( R_{r\overline{2}1\overline{2}} \theta_{1r} -  R_{1\overline{r}1\overline{2}} \theta_{r2} \big) = 2D (\theta_{11} - \theta_{22}). $$
So if $D\neq 0$, then $\theta_{11}=\theta_{22}=0$. Applying $\nabla R=0$ to $R_{1\overline{1}1\overline{2}}$, we get
$$ 0 = dR_{1\overline{1}1\overline{2}} =  \sum_{r=1}^5 \big( 2 R_{r\overline{1}1\overline{2}} \theta_{1r} -  R_{1\overline{r}1\overline{2}} \theta_{r1} - R_{1\overline{1}1\overline{r}} \theta_{r2} \big) = -3a\theta_{12} - D \theta_{21}. $$
This gives us $\theta_{12} = -\frac{D}{3a}\theta_{21}$. If $D=0$, then $\theta_{12}=0$, so
$$ \Theta_{12}=d\theta_{12} - (\theta_{11} -\theta_{22})\wedge \theta_{12} = 0,$$
hence $-a=R_{1\overline{1}2\overline{2}} = R_{2\overline{1}1\overline{2}} = \Theta_{12}(e_2, \overline{e}_1) = 0$, a contradiction. If  $D\neq 0$, then we have $\theta_{11}=0$ and $\theta_{12}$ proportional to $\theta_{21}$, so
$$ \Theta_{11} = d\theta_{11} - \theta_{12} \wedge \theta_{21} = 0,$$
which leads to $a=R_{1\overline{1}1\overline{1}}=0$, a contradiction. So either way, the assumption that $R\neq 0$ leads to a contradiction, thus the manifold $M$ must be Chern flat. This completes the proof of the case.

\vspace{0.3cm}

 {\bf Case 3: $\dim \mathcal{W}=1$.}

 Let $e=\{ e_1, \ldots , e_5\}$ be a local unitary frame so $\nabla e_5=0$ and $e_5$ spans $\mathcal{W}$. By $\eta=0$, we have $T^5_{i5}=0$. So $T^5$ is a parallel skew-symmetric form on $\mathcal{N}\cong {\mathbb C}^4$. Denote by $A$ the $4\times 4$ matrix $(T^5_{ij})$, where $1\leq i,j\leq 4$. If we  change the basis $\{e_1, \ldots , e_4\}$ by  a unitary matrix $B$, then $A$ is changed to $BA \,^t\!B$. As is well known, we can choose $e$ so that $A$ takes the following canonical form:
 $$ A = \left[ \begin{array}{cc} aE & 0 \\ 0 & bE \end{array} \right] , \ \ \ \ \ \ \mbox{where} \ \ \ \ \ E = \left[ \begin{array}{cc} 0 & 1\\ -1 & 0 \end{array} \right] $$
 and $a\geq b\geq 0$. In other words, under the frame $e$, the parallel $2$-form $\Phi = \sum_{i,j}T^5_{ij}\varphi_i\wedge \varphi_j$ is given by
 $$ \Phi = 2a \varphi_1\wedge \varphi_2 + 2b \varphi_3 \wedge \varphi_4.$$
 We have $b>0$ since otherwise $\mathcal{N}_0$ will be non-trivial. Also, $a^2+b^2=\frac{1}{2}|T|^2$ is a constant. By looking at the norm of the parallel $(4,0)$-form $\Phi \wedge \Phi $, we know that $|ab|=ab$ is a constant, so both $a$ and $b$ are constants. By $\nabla T=0$, we have
 $$ 0=dT^5_{13} = \sum_{r=1}^5 \big( T^5_{r3}\theta_{1r} + T^5_{1r}\theta_{3r} - T^r_{13}\theta_{r5} \big) = -b \theta_{14} + a \theta_{32}.$$
 Similarly, using $T^5_{24}$ we get $b\theta_{23}-a\theta_{41}=0$, and taking conjugate, it yields $-b\theta_{32}+a\theta_{14}=0$. So if $a>b$, then we must have $\theta_{14} = \theta_{23}=0$. Similarly, we get $\theta_{13}=\theta_{24}=0$. That is, the matrix $\Theta$ is block diagonal. Each block is Ricci flat, and by the same argument as in the proof of Theorem \ref{thm-dim3} or in the case 2(2) above, we can utilize the fact that $T$ and $R$ are parallel to conclude that $R=0$.

 Now we are only left with the case that $a=b$. In this case we have $4a^2= |T|^2$, so $a>0$ is a global constant. $T$ gives a holomorphic symplectic structure on $\mathcal{N}$, and locally there always exists unitary frame $e$ so that $\nabla e_5=0$ and $A=(T^5_{ij})$ takes the form
 $$ A = a \left[ \begin{array}{cc} E & 0 \\ 0 & E \end{array} \right] .$$
 We will call such a frame {\em admissible}. Under such a frame, utilizing the fact that $\nabla T=0$, we derive as before that the connection matrix $\theta$ obeys the following symmetries:
 \begin{equation}
 \theta_{11}+\theta_{22}=0, \ \ \ \theta_{33}+\theta_{44}=0, \ \ \ \theta_{14}-\theta_{32}=0,\ \ \ \theta_{13}+\theta_{42}=0.
 \end{equation}
By the structure equation $\Theta = d\theta -\theta \wedge \theta$, it yields that the curvature matrix $\Theta$ under $e$ also satisfy the same symmetries:
\begin{equation}
 \Theta_{11}+\Theta_{22}=0, \ \ \ \Theta_{33}+\Theta_{44}=0, \ \ \ \Theta_{14}-\Theta_{32}=0,\ \ \ \Theta_{13}+\Theta_{42}=0,
 \end{equation}
or equivalently, the (upper $4\times 4$ block of the) matrix $\Theta$ is in the form
\begin{equation}
 \Theta = \left[ \begin{array}{cccc} x & y & \alpha & \beta \\ -\overline{y} & \overline{x} & - \overline{\beta} & \overline{\alpha} \\ - \overline{\alpha} & \beta & z & w \\
 -\overline{\beta} & - \alpha & - \overline{w} & \overline{z}\end{array} \right] ,  \label{eq:fivefold}
 \end{equation}
where $\overline{x} = -x$ and $\overline{z} = -z$.

\vspace{0.3cm}

We do not know if in this last situation the Chern curvature could be non-vanishing or not. If it could, then the Hermitian fivefold $M^5$ would have $\dim \mathcal{W}=1$, $T$ gives a parallel holomorphic symplectic structure on $\mathcal{N}=\mathcal{W}^{\perp}$, and with the Chern curvature obeying the above special type of unitary and symplectic symmetry, which when interchanging the second and third elements of the basis, the curvature matrix will be a $4\times 4$ skew-Hermitian matrix in the following form:
 \begin{equation} \label{USP4}
 \left[ \begin{array}{cc} A & B \\ - \overline{B} & \overline{A}  \end{array} \right],  \ \ \mbox{where} \ \  \overline{ ^t\!A} = -A, \ ^t\!B =B.
 \end{equation}

Such a Ricci-flat fivefold, if exists (and not Chern flat), will certainly be a very interesting example of non-K\"ahler Calabi-Yau manifold, with rich and highly restrictive metric structure. In summary, the discussion of this section gives us the following

\begin{theorem}\label{thm-dim5}
Let $(M^5,g)$ be a CAS manifold. Then either it is trivial (i.e., each de Rham factor is either K\"ahler or Chern flat), or (if exists) it is an irreducible manifold with $\dim \mathcal{W}=1$ such that $T$ gives a parallel holomorphic symplectic structure on $\mathcal{N}=\mathcal{W}^{\perp}$ and the Chern curvature $R$ is non-flat and obeying the above $\mathsf{USp}_4({\mathbb C})$-symmetry, namely the symmetry given in (\ref{eq:fivefold}) or (\ref{USP4}).
\end{theorem}

\section{CAS manifolds in general dimensions}

In this section, we consider CAS manifolds in general dimensions. First of all, the statements of Theorems \ref{thm-dim3}, \ref{thm-dim4} and \ref{thm-dim5}  can be slightly generalized, to conclude that any CAS manifold whose image distribution $\mathcal{W}$ has codimension at most $3$ must be trivial:

\begin{theorem}\label{thm-imagecodim}
Let $(M^n,g)$ be a CAS manifold. If the image distribution $\mathcal {W}$ of the Chern torsion has codimension at most $\,3$, then $M$ is trivial, namely, each de Rham factor of $M$ is either K\"ahler (hence Hermitian symmetric) or Chern flat.
\end{theorem}

\begin{proof}
Denote by $p$ the codimension of $\mathcal {W}$. By Theorem \ref{thm-imagecodim1} we already see that if $p=1$, then the universal cover of $M$ is holomorphically isometric to the product of  a $1$-dimensional K\"ahler factor (a complex space form) with Chern flat factor (a complex Lie group equipped with a left-invariant metric). Let us assume that $(M^n,g)$ is without K\"ahler de Rham factor. By Theorem \ref{thm-N0}, this means that the kernel $\mathcal{N}_0$ of the Chern torsion is trivial. Since $\mathcal{W}$ is contained in the kernel of the Chern curvature tensor $R$, around any given point in $M$, we may take a local unitary frame $e$ so that $\{ e_1, \ldots , e_p\}$ spans $\mathcal{N}$, $\{ e_{p+1}, \ldots , e_n\}$ spans $\mathcal{W}$, and $\nabla e_{\alpha}=0$ for each $p<\alpha \leq n$.

First consider the $p=2$ case. In this case if there is a local vector field in $\mathcal{N}$ that is parallel under $\nabla$, then by the parallelness of $\mathcal{N}$ there will be a local parallel tangent frame, thus $M$ is Chern flat. So under the assumption that $M$ is not Chern flat, we must have $\psi_X=0$ for all parallel $X\in \mathcal {W}$, where $\psi_X$ is the $(1,0)$-form on $\mathcal{N}$ defined by $\langle T(\cdot , X), \overline{X} \rangle$. This means $T^{\alpha}_{1\beta}=T^{\alpha}_{2\beta}=0$ for any $3\leq \alpha, \beta \leq n$. So by $\mathcal{N}_0 =0$, we know that there exists some $3\leq \alpha \leq n$ such that $T^{\alpha}_{12}\neq 0$. Fix such an $\alpha$. Note that the value $|T^{\alpha}_{12}|$ is a constant. By rotating $e_1$ if necessary, we may assume that $T^{\alpha}_{12}=|T^{\alpha}_{12}|$. Thus by applying $\nabla T=0$ to $T^{\alpha}_{12}$, we get $T^{\alpha}_{12}(\theta_{11}+\theta_{22})=dT^{\alpha}_{12} =0$, hence $\theta_{11}+\theta_{22}=0$ under this choice of local frame $e$, and $M$ is Chern Ricci flat. Following the same analysis on the Chern curvature tensor $R$ as in the proof of Theorem \ref{thm-dim3}, we conclude that $R$ must be zero.

Next let us consider the $p=3$ case. In this case $\mathcal{N}\cong {\mathbb C}^3$. For each $4\leq \alpha \leq n$,
$T^{\alpha}_{\ast ,\ast}$ gives us a parallel, skew-symmetric form on $\mathcal{N}$. Either it is zero or it has a one-dimensional kernel, in the direction $Y_{\alpha}$. That is, for each $4\leq \alpha \leq n$, either $\langle T(\mathcal{N}, \mathcal{N}), \overline{e}_{\alpha} \rangle =0$, or we have a parallel direction $Y_{\alpha}\in \mathcal{N}$. Similarly, for any given $4\leq \alpha, \beta \leq n$, consider the $1$-form $\psi_{\alpha\beta}$ on $\mathcal{N}$ defined by $\langle T(\cdot , e_{\beta}), \overline{e}_{\alpha}\rangle $. It is parallel since $T$ and each $e_{\alpha}$ are so. Either it is identically zero, which means $T^{\alpha}_{i\beta}=0$ for each $1\leq i\leq 3$, or it corresponds to a parallel direction $Z_{\alpha\beta}$ in $\mathcal{N}$. Clearly, $Z_{\alpha\beta}=\sum_{i=1}^3 \overline{T^{\alpha}_{i\beta}}e_i$. Note that the case
$$ T^{\alpha}_{ij}=0, \ \ \ T^{\alpha}_{i\beta}=0, \ \ \ \ \ \forall \ 1\leq i,j\leq 3, \ \forall \  4\leq \alpha, \beta \leq n $$
cannot occur, since we assumed $\mathcal{N}_0 =0$. Therefore there will be some parallel directions $Y_{\alpha}$ or $Z_{\alpha\beta}$ in $\mathcal{N}$. If we have more than one local parallel directions in $\mathcal{N}$, then we would have a parallel frame thus $M$ is Chern flat. So we may assume that there is one parallel direction, which we take as our $e_3$, and all $Y_{\alpha}$ and $Z_{\alpha\beta}$, if exists, are proportional to $e_3$. This means that we have $\nabla e_3=0$ and
$$ T^{\alpha}_{13}=T^{\alpha}_{23}=0, \ \ \ T^{\alpha}_{1\beta}=T^{\alpha}_{2\beta} =0,  \ \ \ \ \forall \ 4\leq \alpha , \beta \leq n. $$
So the only possibly non-zero components of $T$ are $T^{\alpha}_{12}$ and $T^{\alpha}_{3\beta}$. Now if we follow the same argument as in {\it Subcase (2)} of {\bf Case 2} of the proof of Theorem \ref{thm-dim5}, we conclude in exactly the same way that $M$ must be Chern flat. This completes the proof of Theorem \ref{thm-imagecodim}.
\end{proof}

Next we show that any CAS manifold without K\"ahler de Rham factor must be Chern Ricci flat. A key technical point used repeatedly in the proof is the following observation:

\begin{lemma}\label{lemma:61} Let $(M^n,g)$ be a Hermitian manifold with a CAS structure. Assume that $Z$ is a parallel holomorphic vector field on $M$. Define its dual $(1, 0)$-form as  $\xi^Z(\cdot) \doteqdot \langle \cdot, \overline{Z}\rangle$ and the associated torsion $(2, 0)$-form $\tau^{Z}(\cdot, \cdot)\doteqdot \langle T(\cdot, \cdot), \overline{Z}\rangle$.
Then $\xi^Z$ is a holomorphic $1$-form, with $d \xi^Z=\tau^Z$, and both $\xi^Z$ and $\tau^Z$ are  parallel with respect to the Chern connection.
\end{lemma}

\begin{proof} Let $W,U$ be local type $(1,0)$ vector fields. Since $(T(W,\bar{U})=0$, by the general formula (\ref{eq:general1}) we have
\begin{eqnarray*}
\bar{\partial} \xi^Z(W, \overline{U})\ = \ d\xi^Z(W, \overline{U})&=&(\nabla_W\xi^Z)(\overline{U})-(\nabla_{\overline{U}}\xi^Z)(W) \ = \ -(\nabla_{\overline{U}}\xi^Z)(W)  \\
&=&-\overline{U}\langle W, \overline{Z}\rangle +  \langle \nabla_{\overline{U}}W, \overline{Z}\rangle  \ = \ -\langle W, \nabla_{\overline{U}}\overline{Z}\rangle =0.
\end{eqnarray*}
In the last line we used the fact that $\overline{Z}$ is parallel. Replacing $\overline{U}$ by $U$ we have
\begin{eqnarray*}
\partial \xi^Z(W, U)\ = \ d\xi^Z(W, U)&=&(\nabla_W\xi^Z)(U)-(\nabla_{U}\xi^Z)(W)+\langle T(W, U), \overline{Z}\rangle\\
&=& \langle T(W, U), \overline{Z}\rangle \ = \ \tau^Z(W,U).
\end{eqnarray*}
So $\xi^Z$ is a holomorphic $1$-form, $d\xi^Z=\tau^Z$, and $\tau^Z$ is a $d$-exact holomorphic $2$-form. To see that $\xi^Z$ is parallel, for any (local) complex vector field $X,U$ with $U$ being of type $(1,0)$, we have
$$
(\nabla_X \xi^{Z})(U)=X\langle U, \overline{Z}\rangle -\langle \nabla_X U, \overline{Z}\rangle = \langle U, \nabla_X\bar{Z} \rangle = 0.
$$
So $\nabla \xi^Z=0$. Similarly, $\nabla \tau^Z =0$.
\end{proof}

\begin{proof}[{\bf Proof of Theorem \ref{thm-Ricciflat}.}]

Let $(M^n,g)$ be a Hermitian manifold with a CAS structure and without any K\"ahler de Rham factor. This means that $\mathcal{N}_0=0$ by Theorem \ref{thm-N0}. We have orthogonal decomposition $T^{1,0}M = \mathcal{N} \oplus \mathcal{W}$ into $\nabla$-parallel distributions where $\mathcal{W}$ is  the image of $T$. We already know that $\mathcal{W}$ is integrable and is contained in the kernel of the Chern curvature tensor $R$, hence is a holomorphic foliation with Chern flat leaves. Let $e$ be a local unitary frame so that $\{e_i\}_{1\leq i\leq r}$ spans $\mathcal{N}$ and $\{e_{\alpha}\}_{r< \alpha \leq n}$ spans $\mathcal{W}$, with $\nabla e_{\alpha} = 0$ for each $r+1\leq \alpha \leq n$. Lifting to the universal cover $\widetilde{M}$ if necessary, we may assume that the parallel frame $\{ e_{\alpha}\}_{r< \alpha \leq n}$ of $\mathcal{W}$ is globally defined (since by the holonomy theorem of Ambrose-Singer, the action of any holonomy restricted to $\mathcal{W}$ is trivial). Let $\varphi$ be the unitary coframe dual to $e$. Then the Chern connection matrix $\theta$ under $e$ satisfies $\theta_{\alpha \ast }=0$, so we have
$$ d\varphi_i = -\sum_{j=1}^r \theta_{ji}\wedge \varphi_j, \ \ \ d\varphi_{\alpha} = \tau^{\alpha} = \sum_{a,b=1}^n T^{\alpha}_{ab}\varphi_a\wedge \varphi_b. $$
Note that while those $\varphi_i$ are only defined locally, each $\varphi_{\alpha}$ is a globally defined parallel holomorphic $(1, 0)$-form on $\widetilde{M}$ by Lemma \ref{lemma:61}, and each $\tau^{\alpha}$ is also a globally defined holomorphic, parallel $(2,0)$-form and is $d$-exact. Write $\tilde{\tau}^{\alpha}=\sum_{i,j=1}^r T^{\alpha}_{ik}\varphi_i\wedge \varphi_j$ for the part of $\tau^{\alpha}$ modulo $\mbox{span}\{ \varphi_{\beta}\mid \, r+1\leq \beta \leq n\}$.

For each $\alpha$, consider the skew-symmetric bilinear form $A^{\alpha}:= T^{\alpha}|_{\mathcal{N}\times \mathcal{N}}= \langle T(\cdot , \cdot), \bar{e}_{\alpha}\rangle $ on $\mathcal{N}$. Clearly it is parallel with respect to $\nabla$. Under the local unitary  frame $\{ e_i\} $ of $\mathcal{N}$, $A^{\alpha}$ is represented by a skew-symmetric $r\times r$ matrix which we will still denote by the same letter. Assume $2k$ is the maximum rank of $A^X$ for all $X\in \mathcal{W}$, and with a constant unitary change of $\{ e_{\alpha}\}$ if necessary, we may assume that $ \mbox{rank}(A^n) = 2k$. As is well-known, for any given skew-symmetric $r\times r$ matrix $A^n$, there exists a unitary matrix $U$ such that
$$ ^t\!U A^n U = \left[ \begin{array}{llll} a_1E & & & \\ & \ddots & & \\  & & a_kE &  \\  & & &  0_{r-2k}  \end{array} \right] ,   \ \ \ \ \ E = \left[ \begin{array}{ll } 0 & 1 \\ -1 & 0 \end{array} \right], \ \ \ \ a_1\geq \cdots \geq a_k >0.$$
Since $-a_i^2$ are eigenvalues of the parallel Hermitian form $A^n\overline{A^n}$, so all $a_i$ are constants. By a unitary change of $\{ e_i\}$ if needed, we may assume that $A^n$ is in the above block-diagonal form. Fix this frame $e_i$ now, and let $\mathcal{N}_1$, $\mathcal{N}_2$ be the space spanned by $\{ e_1, \ldots , e_{2k}\}$ and $\{ e_{i'} \mid 2k+1\leq i'\leq r\}$, respectively. We have the orthogonal decomposition  $\mathcal{N}= \mathcal{N}_1\oplus \mathcal{N}_2$, and each factor is $\nabla$-parallel since $\mathcal{N}_2$ is the kernel space of the parallel form $A^n$. For each $r+1\leq \alpha \leq n-1$, write
$$ A^{\alpha} = \left[ \begin{array}{cc } \ast & B \\ -\,^t\!B & C \end{array} \right] $$
where $C$ is a $(r-2k)\times (r-2k)$ matrix. Since $tA^n+ A^{\alpha} = A^{te_n+e_{\alpha}}$ has rank at most $2k$ for any $t\in {\mathbb C}$, we know that $B=C=0$ in the above. So for each $i'$ in the $\mathcal{N}_2$ block, the only possibly non-zero torsion components involving $i'$ are $T^{\alpha}_{i'\beta}$, where $\alpha , \beta \in \mathcal {W}$. For any fixed $\alpha$, $\beta$, consider the type $(1,0)$ vector field
$$ X^{\alpha}_{\beta} = \sum_{i'=2k+1}^r \overline{T^{\alpha}_{i'\beta}}\,e_{i'}.$$
It is a globally defined vector field in the distribution $\mathcal{N}_2$, and is parallel as it is dual to the $(1,0)$-form $\psi$ on $\mathcal{N}_2$:  $\psi(\cdot )= \langle T(\cdot , e_{\beta}), \overline{e}_{\alpha} \rangle $, in the sense that $\psi (Y)= \langle Y, \overline{X^{\alpha}_{\beta}}\rangle$ for any $Y$. Since $\psi$ is clearly parallel, so is $X^{\alpha}_{\beta}$. Clearly, these $\{X^{\alpha}_{\beta}\}$ for all $r\le \alpha$, $\beta\le n$ span the entire $\mathcal{N}_2$, as otherwise we would have a direction $i'$ in $\mathcal{N}_2$ such that $T^{\ast}_{i'\ast}=0$, hence $e_{i'}\in \mathcal{N}_0$, contradicting with our assumption that $M^n$ is without K\"ahler de Rham factor. This shows that the  distribution $\mathcal{N}_2$ admits a parallel frame, thus belonging to the kernel of the curvature tensor $R$.

Now let us further assume that each $e_{i'}$, $2k+1\leq i' \leq r$ is also parallel. Thus we have $\theta_{i'\ast}=0$ and hence $d\varphi_{i'}=0$, namely, each $\varphi_{i'}$ is a $d$-closed holomorphic $1$-form. We obtain locally
$$ \tilde{\tau}^n = \sum_{i,j=1}^r T^n_{ij} \varphi_i \wedge \varphi_j = 2a_1\varphi_1\wedge\varphi_2 + 2a_2\varphi_3\wedge \varphi_4 + \cdots + 2a_k \varphi_{2k-1}\wedge \varphi_{2k}.$$
Now consider the globally defined holomorphic $(n, 0)$-form
$(\tau^n)^k \wedge \bigwedge_{i'=2k+1}^r \varphi_{i'} \wedge \bigwedge_{\alpha=r+1}^n \varphi_{\alpha}$.
We have that on $\widetilde{M}$
$$ \Psi\doteqdot (\tau^n)^k \wedge \bigwedge_{i'=2k+1}^r \varphi_{i'} \wedge \bigwedge_{\alpha=r+1}^n \varphi_{\alpha} = (\tilde{\tau}^n)^k \wedge \varphi_{2k+1} \wedge \cdots \wedge \varphi_n = (2^kk! \,a_1\cdots a_k)\, \varphi_1 \wedge \cdots \wedge \varphi_n. $$
It is easy to see that $\Psi \overline{\Psi}$ equals to a non-zero constant multiple of the volume form. This implies that the Chern Ricci curvature of $M$ must vanish, and we have completed the proof of Theorem \ref{thm-Ricciflat}. It is also clear from the construction that $\Psi$ is parallel. As in the proof of Theorem \ref{thm-imagecodim}, if $k=1$, then $(M, g)$ is Chern flat.
\end{proof}

We remark  that, if the universal cover $\widetilde{M}$ is compact, then the above  $\{\varphi_{i'}\}$ terms cannot exist, as they are $d$-closed holomorphic $1$-forms which must be $d$-exact by topological consideration, hence $\varphi_{i'}=df$ for some global holomorphic function $f$ on $\widetilde{M}$ which must be constant. Therefore, if  $n$ is odd then $\dim(\mathcal{W})$ must be odd, while if  $n$ is even then  $\dim(\mathcal{W})$ must be even. Similarly since $\bar{\partial}\varphi_\alpha =0$, we have that $\dim(H_{\bar{\partial}}^{1, 0}(\widetilde{M}))\ge 1$, if the CAS manifold $M$ is not locally Hermitian symmetric.

\section{Homogenous manifolds with a CAS structure}

In this section, we will give some general discussion of manifolds with a CAS structure in relation to its local homogeneity. By the virtue of Theorem 7.4 of \cite{KN}, for a manifold with the CAS structure, given any two points $x, y\in M$ and any curve $\gamma$ joining them, one can find a neighborhood $U$ of $x$ and a local isometric holomorphic map $f: U\to M$ such that $f(x)=y$ such that $df|_{x}$ coincides with the parallel transport along $\gamma$.

First let us give a proof of Corollary \ref{cor-homog}. For this we will need a Bochner formula which should be well-known to experts. Let us recall the notion of first and second Ricci curvature. Suppose that $(M^n,g)$ is a Hermitian manifold, with $\nabla$ its Chern connection and $R$ the Chern curvature tensor. The {\em first (second) Chern Ricci} curvature tensor are defined respectively by
$$ Ric^{(1)}_{i\bar{j}}=\sum_{k=1}^n  R_{i\bar{j}k\bar{k}} , \ \ \ \ Ric^{(2)}_{i\bar{j}}=\sum_{k=1}^n R_{k\bar{k}i\bar{j}}$$
under any unitary frame. Suppose $Z$ is a holomorphic vector field on $M^n$, then we have the following well-known Bochner formula
\begin{equation}
\Delta |Z|^2 = |\nabla Z|^2 - Ric^{(2)}_{Z\bar{Z}}.  \label{eq:Bochner}
\end{equation}
Here the left hand side is the complex Laplacian, namely, the trace of $\sqrt{-1}\partial \overline{\partial} |Z|^2$ with respect to the K\"ahler form $\omega$ of $g$. To verify the formula, we just need to recall the following commutation formula
$$ Z_{i,k\bar{\ell}} - Z_{i ,\bar{\ell}k} = - \sum_{r=1}^n Z_r R_{k\bar{\ell}i\bar{r}} $$
under any unitary frame $e$, where $Z=\sum_i Z_ie_i$ and the indices after comma stand for covariant derivatives with respect to $\nabla$. In particular, $\sum_k Z_{i,k\bar{k}} = - \sum_r Z_r Ric^{(2)}_{i\bar{r}}$. Fix any $p\in M$, we can always choose a local unitary frame $e$ so that the connection matrix of $\nabla$ vanishes at $p$. Since $Z$ is holomorphic, we always have $Z_{i,\bar{j}}=0$, thus
\begin{eqnarray*}
 \Delta |Z|^2 & = &  g^{i\bar{k}}\frac{\partial^2}{\partial z^i \partial \bar{z}^k}  |Z|^2 \ = \  \sum_k (|Z|^2)_{,k\bar{k}} \ = \ \sum_{i,k} (Z_{i,k}\overline{Z}_i)_{,\bar{k}}  \\ & = &  \sum_{i,k} |Z_{i,k}|^2 + \sum_{i,k} Z_{i,k\bar{k}} \overline{Z}_i \ = \  |\nabla Z|^2 - Ric^{(2)}_{Z\bar{Z}}. \\
\end{eqnarray*}
So we have verified the Bochner formula (\ref{eq:Bochner}). As an immediate consequence, we have

\begin{lemma}
If $(M^n,g)$ is a compact Hermitian manifold whose second Chern Ricci is non-positive, then any holomorphic vector field $Z$ on $M$ must be parallel with respect to the Chern connection $\nabla$.
\end{lemma}

By a classic result of Bochner and Montgomery \cite{BM}, for any given compact complex manifold $M^n$, its automorphism group  $\mbox{Aut}(M)$ (namely, the group of biholomorphisms from $M$ onto itself) is a complex Lie group, and its Lie algebra is the Lie algebra of holomorphic vector fields on $M$. The latter is equal to $H^0(M,T^{1,0}M)$, the space of global holomorphic sections of the holomorphic tangent bundle of $M$.

A compact complex manifold $M$ is said to be {\em homogeneous} if $\mbox{Aut}(M)$ acts transitively on $M$. It is said to be {\em almost homogeneous} if $\mbox{Aut}(M)$ has an open orbit in $M$. In particular, at every point $p$ in a homogeneous manifold (or at a generic point $p$ in an almost homogeneous manifold), the global holomorphic sections span the tangent space $T^{1,0}_pM$. That is, there exists $Z_1, \ldots , Z_n \in H^0(M,T^{1,0}M)$ such that $\{ Z_1(p), \ldots , Z_n(p)\}$ is linearly independent. The aforementioned Bochner formula leads to the following

\begin{proposition}\label{prop:71}
Let $M^n$ be a compact homogeneous (or almost homogeneous) complex manifold. If $g$ is any Hermitian metric on $M^n$ with non-positive second Chern Ricci, then $g$ is Chern flat, thus $M$ is the quotient of a complex Lie group by a discrete subgroup.
\end{proposition}

\begin{proof} By the assumption we may fix a point  $p\in M$ and pick $Z_1, \cdots, Z_n$, global holomorphic vector fields on $M^n$ such that they are linear independent at $p$. By the Bochner formula (\ref{eq:Bochner}), we know that each $Z_i$ is parallel with respect to the Chern connection $\nabla$ of $g$, hence we have a parallel frame, thus $g$ is Chern flat. Now the result follows from H.-C. Wang's theorem \cite{Wang-PAMS}.
\end{proof}

It was known that \cite{AK} for a Riemannian homogenous manifold, the only Ricci flat ones are totally flat (in fact isometric to $T^{n-k}\times \mathbb{R}^k$). {\it Does the same result holds for a homogenous complex manifold with the second Ricci curvature (with respect to the Chern connection) vanishing?} The above asserts that it is the case when $M$ is compact. The resolution of this problem is related to the possible example suggested by Section 5.

Now we  prove Corollary \ref{cor-homog} stated in the introduction.

\begin{proof}[{\bf Proof of Corollary \ref{cor-homog}.}]
Let $(M^n,g)$ be a CAS manifold. Assume also that $M$ is almost homogeneous. Each K\"ahler de Rham factor of $M$ is necessarily Hermitian symmetric. Let us denote by $N_2$ (or $N_3$) the product of all K\"ahler de Rham factors that are Hermitian symmetric of non-compact type (or compact type), then we have the decomposition $\widetilde{M}=N_1\times N_2\times N_3$  on the universal cover level where $N_1$ is the product of all non-K\"ahler de Rham factors. On $M^n$,  the holomorphic tangent bundle is the  direct sum $T^{1,0}M = T_1 \oplus T_2\oplus T_3$. Since a CAS metric is Chern K\"ahler-like, its first and second Chern Ricci curvature coincide. By Theorem \ref{thm-Ricciflat}, we know that $N_1$ is Chern Ricci flat. As is well-known, the $N_2$ factor has negative Ricci curvature, so $N_1\times N_2$ has non-positive second Chern Ricci curvature. So for any global holomorphic vector field $Z\in H^0(M, T_1\oplus T_2)$, $Ric^{(2)}(Z,\bar{Z}) \leq 0$, thus by the Bochner formula (\ref{eq:Bochner}) we know that $\nabla Z=0$. At a generic point $p\in M$, $T_1\oplus T_2$ is generated by global sections, hence has a parallel frame. This means that $N_1\times N_2$ is Chern flat. In particular,  any Hermitian symmetric space of non-compact type does not appear as de Rham factors of $M$. Since $N_3$ is compact and simply-connected, we have $\widetilde{M}=N_1\times N_3$ where $N_1$ is Chern flat. We have thus completed the proof of Corollary \ref{cor-homog}.
\end{proof}



Next we consider a Hermitian  manifold $M$ with a CAS structure. The group of holomorphic isometries of the universal cover $\widetilde{M}$ is always a Lie group, but it might no longer be a complex Lie group. So $\widetilde{M}$ admits an effective and transitive action of a Lie group $G$ which is isometric and holomorphic. We may further assume that $G$ is connected, and  $\widetilde{M}$ can be identified with the coset space $G/H$, where $H$ is the isotropic group fixing one point $o\in \widetilde{M}$. Note that $H$ is compact since it is a closed subgroup of the unitary group, and $M=\Gamma \backslash \widetilde{M}$ where $\Gamma \subseteq G$ is a discrete subgroup.


In the following, let $(M^n,g)$ be a Hermitian manifold. When $X$ is a vector field which preserves the Hermitian metric and the holomorphic structure, the operator $A_X=L_X-\nabla_X$ satisfies some special properties. In this section $X, Y, S$ are real vector fields.

\begin{proposition}\label{prop:62} Let $(M^n, g)$ be a Hermitian manifold. Assume that $X$ is a  real holomorphic vector field (namely, $L_XJ=0$, or equivalently that the $1$-parameter family of diffeomorphisms generated by $X$ preserve the holomorphic structure $J$) which is also Killing (that is, $L_X g=0$). Then
$\nabla_Y(A_X)=R_{X Y}$.
\end{proposition}

\begin{proof} Let $\varphi_t$ be the 1-parameter family of diffeomorphisms generated by $X$. It then preserves the Chern connection due to that it is the unique one which is compatible both with the metric and complex structure $J$. Hence
\begin{eqnarray*}
L_X \cdot \nabla_Y S&=&\lim_{t\to 0} \frac{1}{t}\left(\nabla_Y S -d\varphi_t(\nabla_Y S)\right)\\
&=& \lim_{t\to 0} \frac{1}{t}\left((\nabla_Y S-\nabla_{d\varphi_t( Y)}S)+(\nabla_{d\varphi_t (Y)}S-\nabla_{d\varphi_t (Y)}d\varphi_t(S))\right)\\
&=&\nabla_{[X, Y]} S+\nabla_Y \cdot  L_X S.
\end{eqnarray*}
Namely we have that
\begin{equation}\label{eq:K0stant1}
L_X \cdot \nabla_Y -\nabla_Y \cdot L_X=\nabla_{[X, Y]}.
\end{equation}
Now using (\ref{eq:K0stant1}) we calculate
\begin{eqnarray*}
R_{X, Y}&=&\nabla_X \cdot \nabla_Y -\nabla_Y\cdot \nabla_X -\nabla_{[X, Y]}\\
&=&(L_X-A_X)\cdot \nabla_Y -\nabla_Y (L_X-A_X)-\nabla_{[X, Y]}\\
&=&-A_X \cdot \nabla_Y +\nabla_Y\cdot  A_X.
\end{eqnarray*}
This proves the claimed identity.
\end{proof}

\begin{lemma}\label{lemma:kostant1}
(i) If $X$ is Killing, then $A_X$ is skew-symmetric;

(ii) If $X$ is real holomorphic, then $[A_X, J]=0$;

(iii) For any $X$, $Y$, it holds
\begin{equation}\label{eq:Konstant2}
A_X(Y)=-\nabla_Y X -T(X, Y).
\end{equation}
In particular, when $X$ is real holomorphic, $\nabla_{JY} X=J\nabla_Y X$.

(iv) If $X$ is real holomorphic and Killing, then $\iota_XT$ is a $d$-closed $1$-form valued in $T'M$;

(v) If the Hermitian $(M, g)$ has a  CAS structure, then $A_X$ satisfies the Jacobi equation
\begin{equation}\label{eq:kostant3}
\nabla^2_{S, Y}(A_X)\doteqdot \nabla_S \cdot \nabla_Y (A_X)-\nabla_{\nabla_S Y} A_X=-R_{A_X(S), Y}, \mbox{ namely } \nabla^2_{(\cdot), (\cdot)} [A_X(\cdot)] = -R_{[A_X(\cdot)], (\cdot)}(\cdot).
\end{equation}
\end{lemma}
\begin{proof} Direct calculation shows that
\begin{eqnarray*}
\nabla_S \cdot \nabla_Y (A_X)-\nabla_{\nabla_S Y} A_X&=&\nabla_S (R_{X, Y})-R_{X, \nabla_S Y}\\
&=& R_{\nabla_S X, Y}=-R_{A_X(S)+T(S, X), Y}=-R_{A_X(S), Y}.\end{eqnarray*}
Here we have used $\nabla R=0$ and $R$ vanishes on the image of $T$.
\end{proof}
Proposition \ref{prop:62} and equation (\ref{eq:kostant3}) give the 1st and 2nd order equations that $A_X$ satisfies.

Consider a type $(1,0)$ vector field $Z=X-\sqrt{-1}JX$. By definition, $X$ is real holomorphic if $L_XJ=0$. Since $L_{JX}J - JL_XJ = N_J(X, \cdot )$, which vanishes as $J$ is integrable, we know that $X$ will be real holomorphic when and only when $JX$ is real holomorphic. In this case we say that $Z$ is a holomorphic vector field. It could also be equivalently described as a type $(1,0)$ vector field which could be expressed as $\sum_i f_i \frac{\partial}{\partial z_i}$ with each $f_i$ hlomorphic in local holomorphic coordinates, or equivalently, by the condition that $\nabla_{\overline{W}}Z=0$ for any type $(1,0)$ vector field $W$ on the manifold. In particular, parallel field of type $(1,0)$ are holomorphic.

On the other hand, a real vector field $X$ is Killing if $L_Xg=0$, or equivalently, $\langle \nabla_YX+T(X,Y), B\rangle = -\langle \nabla_BX+T(X,B),Y\rangle $ for any (real) $Y$ and $B$. So if $\nabla X=0$ then the Killing condition is equivalent to
\begin{equation}
\label{eq:Killing}
\langle T(X,Y),B\rangle = -\langle T(X,B),Y\rangle , \ \ \ \forall \mbox{ real vectors} \ Y, B.
\end{equation}
Also since  $JT(X,Y)=T(X,JY)$ for any $Y$, by (\ref{eq:Killing}) we know that if both  $X$ and $JX$ are Killing, then the following holds:
\begin{equation} \label{eq:2Killing}
T(X,\cdot )=0.
\end{equation}
The converse  statement is certainly true, namely, both $X$ and $JX$ are Killing when (\ref{eq:2Killing}) holds. In summary,

{\em On a Hermitian manifold $(M^n,g)$, for any type $(1,0)$ vector field $Z$ which is parallel under the Chern connection $\nabla$, then its real part $X$ and negative imaginary part $JX$ will both be real holomorphic.  $X$ is Killing if and only if (\ref{eq:Killing}) holds, while both $X$ and $JX$ are Killing if and only if (\ref{eq:2Killing}) holds. }

For general $Z=X-\sqrt{-1}JX$ we call $Z$ a Killing vector field if both  $X$ and $JX$ are Killing.

\begin{corollary} For a Hermitian manifold, let $Z=X-\sqrt{-1}JX$ be a parallel vector field. Then $A_Z=-T(Z, \cdot )$ sends $ T'M$ into $ \mathcal{W}$ and $A_Z$ is parallel. If $Z$ is  Killing, then $A_Z= 0$ and $Z$ annihilates the Chern curvature.
\end{corollary}
\begin{proof}
The first part follows from the above discussion. The second part follows from Proposition 3.3 of \cite{Ust} and Proposition \ref{prop:62} above.
\end{proof}

On the other hand,  we also have the following lemma if $X$ and $JX$ are both Killing vectors.

\begin{lemma}For a Hermitian manifold $(M, g)$, if $X$ and $JX$ are both Killing vectors, if additionally $JX$ is real holomorphic then $X$ and $JX$ are both parallel with respect to the torsion connection, namely $A_X, A_{JX}=0$.
\end{lemma}
\begin{proof} Since $T(JX, Y)=JT(X, Y)$, by (iii) of Lemma \ref{lemma:kostant1}, $A_{IX}=JA_X$. Now the two equalities
$$
-\langle A_X Y, B\rangle+\langle Y, A_X B\rangle= \langle A_{JX} JY, B\rangle+\langle JY, A_{JX} B\rangle=0, \langle A_X Y, B\rangle+\langle Y, A_X B\rangle=0
$$
implies that $\langle A_X, Y, B\rangle =0$, hence the first claim. Here we have used $A_{JX}JY=JA_{JX} Y$.
\end{proof}

Holonomy theorem of Ambrose-Singer (cf. Theorem 8.1 of \cite{KN}, Vol 1) implies that $R_{X Y}$ lies inside the Lie algebra of the restricted holonomy group of the Chern connection (cf. Lemma 2.2 of \cite{Ni-21} for a short proof of this part of holonomy theorem). Hence $\nabla_Y (A_X)$ is inside the holonomy algebra of the Chern connection.
Kostant \cite{Kostant} proved that for a compact Riemannian manifold  $A_X$ itself also lies inside the Riemannian holonomy algebra. A natural question is {\it  whether or not  $A_X$ is inside the  holonomy algebra (with respect to the Chern connection, which we abbreviate as c-holonomy) if $X$ is real holomorphic and Killing and  $(M, g)$ is a compact Hermitian manifold.} It can be shown that $A_X$ lies inside the normalizer of the c-holonomy algebra. Let $H_p$ denote the holonomy group centered at $p$.

Now let $\mathcal{F}_p=\{ Z\in T^{1, 0}_pM\, |\, h (Z)=Z, \forall \,h\in H_p\}$. By Ambrose-Singer holonomy theorem and  Lemma \ref{lemma:32-1}, if $(M, g)$ is a Hermitian manifold with a CAS structure then we have  $\mathcal{W}_p\subset \mathcal{F}_p$.

\begin{proposition} Let $(M, g)$ be a Hermitian manifold with a  CAS structure.
Let $\mathcal{F}=\cup_{x} \mathcal{F}_x$ be the sub-bundle of $T^{1, 0}M$. Then $\mathcal{F}$ is a holomorphic integrable foliation. Moreover, if $Z_1, \cdots, Z_r$ is a parallel frame of $\mathcal{F}$, then
$$
[Z_i, Z_j]=c^k_{ij} Z_k
$$
for some constant $c^k_{ij}$. In particular, when $M$ is simply-connected, there exists a complex Lie group $F$ acting almost freely, holomorphically on $M$ such that $T^{1, 0}_x(F\cdot x) =\mathcal{F}_x$.
\end{proposition}
\begin{proof} First note that any $Z\in \mathcal{F}_x$ extends to a parallel vector field $Z$ globally on $M$, which is a section of $T^{1, 0}M$. By the nature of the Chern connection $\nabla$, such a section is holomorphic. If $Z, W$ are two global/local  sections of $\mathcal{F}$ obtained by the parallel extension, we have that
$$
\nabla ([Z, W])=-\nabla T(Z, W)=0.
$$
This proves that $\mathcal{F}$ is integrable and that the structure coefficients are constants.
\end{proof}
Note that by the proof of Theorem \ref{thm-Ricciflat}, $\mathcal{W}\oplus {\mathcal{N}_2}\subset\mathcal{F}$.  The above result is motivated by Theorem 3.7 of \cite{Ust}.

\section*{Acknowledgments} { We would like to take this opportunity to express our gratefulness to the anonymous referee
for the careful reading of our paper and for the numerous corrections/suggestions, which
enhanced the readability of the article.
}



\begin{thebibliography}{99}

\bibitem{AK} D. V.  Alekseevski\u{i} and B. N.  Kimel\'{}fel\'{}d, \textit{ Structure of homogeneous Riemannian spaces with zero Ricci curvature. }(Russian) Funkcional. Anal. i Priloen. \textbf{9} (1975), no. 2, 5--11.

\bibitem {AS} W. Ambrose and I.M. Singer,\textit{On homogeneous Riemannian manifolds,} Duke Math. J., \textbf{25} (1958), 647--669.


\bibitem {BM} S. Bochner and D. Montgomery, \textit{Groups on analytic manifolds,} Ann. of Math. \textbf{48} (1947), 659--669.



\bibitem {Boothby} W. Boothby, \emph{Hermitian manifolds with zero
curvature.} Michigan Math. J., {\bf 5} (1958), no.2, 229--233.

\bibitem{Borel} A. Borel, \textit{K\"ahlerian coset spaces of semisimple Lie groups.} Proc. Nat. Acad. Sci. U.S.A. \textbf{40} (1954), 1147--1151.

 \bibitem{CE}  E.   Calabi and B. Eckmann, \textit{
A class of compact, complex manifolds which are not algebraic.}
Ann. of Math., \textbf{58} (1953), 494--500.




\bibitem {CN} S. Console and L. Nicolodi, \textit{Infinitesimal characterization of almost Hermitian homogeneous
spaces.} Commentationes Mathematicae Universitatis Carolinae, \textbf{40} (1999), 713--721.





\bibitem {Gauduchon1} P. Gauduchon, \textit{La {$1$}-forme de torsion d'une vari\'et\'e hermitienne compacte.} Math. Ann. \textbf{267} (1984), no.4, 495--518.



\bibitem{GR} H. Grauert and R. Remmert, \textit{ \"Uber kompakte homogene komplexe Mannigfaltigkeiten.}
Arch. Math., \textbf{13} (1962), 498--507.


\bibitem{HK} J. Hano and S. Kobayashi, \textit{ A fibering of a class of homogeneous complex manifolds.} Trans. Amer. Math. Soc., \textbf{94} (1960), 233--243


\bibitem  {Kiricenko} V. Kirichenko, \textit{On homogeneous Riemannian spaces with invariant tensor structure.} Soviet Math. Dokl., \textbf{21} (1980), 734--737.

\bibitem{KN} S. Kobayashi and K. Nomizu, \textit{ Foundations of differential geometry. Vol. I, II.} Reprint of the 1963 original. Wiley Classics Library. A Wiley-Interscience Publication. John Wiley \& Sons, Inc., New York, 1996. xii+329 pp.

\bibitem {Kostant} B. Kostant, \textit{ Holonomy and the Lie algebra of infinitesimal motions of a Riemannian manifold.}  Trans. Amer. Math. Soc., \textbf{80} (1955), 528--542.




\bibitem  {Ni-21} L. Ni, \textit{ An alternate induction argument in Simons' proof of holonomy theorem.} Analysis and Partial Differential Equations on Manifolds, Fractals and Graphs (Nankai, 2019, A. Grigoryan, Y. Sun, Eds.) Advances in Analysis and Geometry, \textbf{3} (2021), 443--458.

\bibitem  {NT} L. Nicolodi and F. Tricerri, \textit{On two theorems of I.M. Singer about homogeneous spaces.} Ann.
Global Anal. Geom. \textbf{8} (1990), 193--209.

\bibitem{QW}L. Qin and B.  Wang, \textit{A family of compact complex and symplectic Calabi-Yau
manifolds that are non-K\"ahler.} Geom. Topol., \textbf{22} (2018), 2115--2144.

\bibitem {Sekigawa} K. Sekigawa, \textit{Notes on homogeneous almost Hermitian manifolds.} Hokkaido Math. J., \textbf{7}
(1978), 206--213.


\bibitem {Singer} I.M. Singer, \textit{Infinitesimally homogeneous spaces.}  Comm. Pure Appl. Math.,\textbf{13} (1960), 685--697.

\bibitem{Tits}J. Tits, \textit{ Espaces homog\`enes complexes compacts.} (French) Comment. Math. Helv., \textbf{37} (1962/63), 111--120.

\bibitem {Tricerri} F. Tricerri, \textit{Locally homogeneous Riemannian manifolds,} Rend. Sem. Mat. Univ. Politec.
Torino, \textbf{50} (1994), 411--426.

\bibitem {TV} F. Tricerri and L. Vanhecke,\textit{Homogeneous Structures on Riemannian Manifolds.} London
Math. Soc. Lect. Notes, vol. \textbf{83}. Cambridge Univ. Press, Cambridge (1983).


\bibitem{TY} L. Tseng and S. Yau, \textit{Non-K\"ahler Calabi-Yau manifolds.} String-Math 2011, 241--254,
Proc. Sympos. Pure Math., \textbf{85}, Amer. Math. Soc., Providence, RI, 2012.

\bibitem{Ust} Y.  Ustinovskiy, \textit{
On the structure of Hermitian manifolds with semipositive Griffiths curvature.}
Trans. Amer. Math. Soc., \textbf{373} (2020), no. 8, 5333--5350.

\bibitem{Wang} H.-C.
Wang, \textit{Closed manifolds with homogeneous complex structure.}
Amer. J. Math., \textbf{76} (1954), 1--32.

\bibitem{Wang-PAMS} H.-C.
Wang, \textit{Complex parallisiable manifolds.}
Proc. Amer. Math. Soc., \textbf{5}(1954), 771--776.

\bibitem{Wang-58} H.-C.
Wang, \textit{On invariant connections over principal fiber bundle.}
Nagoya Math. J., \textbf{13} (1958), 1--19.

\bibitem{WYZ} Q. Wang, B. Yang, and F. Zheng, \textit{On Bismut flat manifolds.}
Trans. Amer. Math. Soc., \textbf{373} (2020), no.8, 5747--5772.

\bibitem{Win} J. Winkelmann,\textit{ The classification of three-dimensional homogeneous complex manifolds.} Lecture Notes in Mathematics, \textbf{1602}. Springer-Verlag, Berlin, 1995. xii+230 pp.




\bibitem {YZ} B. Yang and F. Zheng, \textit{ On curvature tensors of Hermitian manifolds.}  Comm. Anal. Geom., \textbf{ 26} (2018), no.5, 1195--1222.
















\end{thebibliography}
\end{document}